\documentclass[a4paper,12pt]{article}
\usepackage{latexsym}
\usepackage{amsmath}
\usepackage{amssymb}
\usepackage{enumerate}
\usepackage{epic,eepic}
\usepackage{color}
\usepackage{bm}
\usepackage{mathrsfs}

\definecolor{refkey}{gray}{0.5}
\definecolor{labelkey}{gray}{0.2}

\usepackage{theorem}
\newtheorem{theorem}{Theorem}[section]
\newtheorem{proposition}[theorem]{Proposition}
\newtheorem{lemma}[theorem]{Lemma}

\theorembodyfont{\rmfamily}
\newtheorem{proof}{\textmd{\textit{Proof.}}}

\newtheorem{remark}[theorem]{Remark}
\newtheorem{example}[theorem]{Example}
\newtheorem{definition}[theorem]{Definition}
\newtheorem{condition}{Condition}

\makeatletter

\@addtoreset{equation}{section}
\makeatother

\newcommand{\qedd}{\hfill \square}
\newcommand{\ve}{\varepsilon}

\newcommand{\lra}{\longrightarrow}
\newcommand{\e}{\mathrm{e}}

\newcommand{\N}{\ensuremath{\mathbb{N}}}

\newcommand{\R}{\ensuremath{\mathbb{R}}}
\newcommand{\Sph}{\ensuremath{\mathbb{S}}}
\newcommand{\bA}{\ensuremath{\mathbf{A}}}
\newcommand{\bV}{\ensuremath{\mathbf{V}}}
\newcommand{\sL}{\ensuremath{\mathsf{L}}}
\newcommand{\sW}{\ensuremath{\mathsf{W}}}
\newcommand{\cJ}{\ensuremath{\mathcal{J}}}
\newcommand{\cH}{\ensuremath{\mathcal{H}}}

\def\diam{\mathop{\mathrm{diam}}\nolimits}

\def\CAT{\mathop{\mathrm{CAT}}\nolimits}

\title{Self-contracted curves in $\CAT(0)$-spaces\\ and their rectifiability}
\author{Shin-ichi OHTA\thanks{Department of Mathematics, Osaka University,
Osaka 560-0043, Japan ({\sf s.ohta@math.sci.osaka-u.ac.jp})}
\thanks{RIKEN Center for Advanced Intelligence Project (AIP),
1-4-1 Nihonbashi, Tokyo 103-0027, Japan}}
\date{\today}
\begin{document}

\maketitle

\begin{abstract}
We investigate self-contracted curves,
arising as (discrete or continuous time) gradient curves of quasi-convex functions,
and their rectifiability (finiteness of the lengths) in Euclidean spaces,
Hadamard manifolds and $\CAT(0)$-spaces.
In the Hadamard case,
we give a quantitative refinement of the original proof of the rectifiability
of bounded self-contracted curves (in general Riemannian manifolds) by Daniilidis et al.
Our argument leads us to a generalization to $\CAT(0)$-spaces
satisfying several uniform estimates on their local structures.
Upon these conditions, we show the rectifiability of bounded self-contracted curves 
in trees, books and $\CAT(0)$-simplicial complexes.
\end{abstract}

\tableofcontents

\section{Introduction}\label{sc:intro}

Let $(X,d)$ be a metric space.
A curve $\xi:[0,\ell) \lra X$ is said to be \emph{self-contracted} if,
for each $T \in (0,\ell)$, the distance $d(\xi(t),\xi(T))$
is non-increasing in $t \in [0,T]$, that is to say,
\[ d\big( \xi(t_2),\xi(t_3) \big) \le d\big( \xi(t_1),\xi(t_3) \big)
 \qquad \text{for all}\ 0 \le t_1 \le t_2 \le t_3 <\ell. \]
We remark that $\xi$ is not necessarily continuous,
and $\ell>0$ can be infinite.
This simple and flexible notion turned out quite useful in the study of the \emph{rectifiability} of curves,
namely the finiteness of the length:
\[ \sL(\xi) :=\sup\bigg\{ \sum_{i=1}^k d\big( \xi(t_{i-1}),\xi(t_i) \big) \,\bigg|\,
 0=t_0<t_1< \cdots <t_k<\ell \bigg\}. \]
The self-contractedness was introduced by Daniilidis et al in \cite{DLS}
to study gradient systems of convex or, more generally, quasi-convex functions
(independently from \cite{MPit,MP} dealing with a related class of curves,
called \emph{self-expanding curves} in \cite{D3L}).
See also \cite{GS} for a related work on surfaces of constant curvature.
We refer to \cite[\S 1]{D3L} for more background information and references.

A fundamental and motivating example of a self-contracted curve
is a gradient curve for a quasi-convex function.
More precisely, given a quasi-convex function $f:\R^n \lra \R$,
each discrete-time gradient curve constructed by the proximal method
(and its piecewise affine continuous extension)
is self-contracted (\cite[Proposition~4.16]{D3L}).
Then the continuous-time gradient curve given as the limit of discrete ones
clearly inherits the self-contractedness.
For gradient curves of a convex function,
the self-contractedness is derived also from the evolution variational inequality.
See \S\S\ref{ssc:gf}, \ref{ssc:gf'} for details in the generalized setting of
$\CAT(0)$-spaces (or $\CAT(1)$-spaces with diameter $<\pi/2$).

The rectifiability is a central subject of the study of self-contracted curves.
It provides a theoretical guarantee on the convergence of gradient curves of quasi-convex functions,
thereby especially quantitative estimates are
of fundamental importance from the viewpoint of optimization theory.
To the best of the author's knowledge, the self-contractedness is the most useful tool to show the rectifiability.
The rectifiability of bounded self-contracted curves
has been established in \cite{D3L} for Euclidean spaces
(see also \cite{LMV} for an independent work on continuous curves),
in \cite{D3R} for Riemannian manifolds,
and in \cite{Le,ST} for finite-dimensional normed spaces.
(We remark that the self-contractedness of gradient curves of convex functions
may fail in normed spaces, see \cite{OSnc} for the failure of the contraction property.)
An important feature of these results is that
the rectifiability holds true only in finite-dimensional spaces.
The estimates of the length indeed depend on the dimension of the space,
and it is known that we can easily construct a counter-example in a Hilbert space
(see Example~\ref{ex:infd}).

Although the self-contractedness is written only in terms of the distance function,
all known results were concerned with manifolds or normed spaces.
The main aim of the present article is to study self-contracted curves in a genuinely metric setting
of \emph{$\CAT(0)$-spaces}
(metric spaces of non-positive sectional curvature in the sense of triangle comparison theorem).
For this purpose, we start with a self-contained review of the Euclidean situation,
followed by the case of \emph{Hadamard manifolds}
(complete, simply-connected Riemannian manifolds of non-positive sectional curvature)
where we give a quantitative refinement of the proof of \cite{D3R} based on comparison theorems
(Theorem~\ref{th:Hrect}).
This argument leads us to three conditions on $\CAT(0)$-spaces under them
the rectifiability is established (see \S \ref{ssc:Crect}).
The conditions introduced in \S \ref{ssc:Crect} are uniform estimates
on the local structures and not fulfilled by general $\CAT(0)$-spaces.
We shall consider \emph{trees} of bounded degrees and \emph{books} consisting of finite sheets
as fundamental examples of $\CAT(0)$-spaces,
and see that these conditions are satisfied (\S\S \ref{ssc:tree}, \ref{ssc:book}).
More generally, $\CAT(0)$-simplicial complexes satisfying quite mild hypotheses
satisfy our conditions (Theorem~\ref{th:simcom}).
This in particular shows that
our argument does not prevent spaces with non-uniform dimensions.
In the final section we will discuss several related further problems.
\bigskip

{\it Acknowledgements}.
I would like to thank Mikl\'os P\'alfia for stimulating discussions.
The author was supported in part by
JSPS Grant-in-Aid for Scientific Research (KAKENHI) 15K04844.

\section{Self-contracted curves in Euclidean spaces}\label{sc:Eucl}

In this section, after a brief explanation of some general properties
of self-contracted curves in metric spaces,
we deal with the fundamental Euclidean setting.
We will give a self-contained proof of the rectifiability of bounded self-contracted curves
along (essentially) the lines of \cite{D3L},
for the sake of completeness as well as a transparent description of
how to generalize it to Hadamard manifolds.

We will denote by $B(x,r)$ (resp.\ $\bar{B}(x,r)$)
the open (resp.\ closed) ball of center $x$ and radius $r>0$.

\subsection{Self-contracted curves}\label{ssc:SC}

Recall from the introduction that
a curve $\xi:[0,\ell) \lra X$ in a metric space $(X,d)$ (with $\ell \in (0,\infty]$)
is said to be \emph{self-contracted} if it satisfies
\begin{equation}\label{eq:SC}
d\big( \xi(t_2),\xi(t_3) \big) \le d\big( \xi(t_1),\xi(t_3) \big)
 \qquad \text{for all}\ 0 \le t_1 \le t_2 \le t_3 <\ell.
\end{equation}
The condition \eqref{eq:SC} is flexible about deformations of the parametrization
(as in (i) of the next lemma).
The following is straightforward from the definition.

\begin{lemma}\label{lm:dcon}
Let $\xi:[0,\ell) \lra X$ be a self-contracted curve.
\begin{enumerate}[{\rm (i)}]
\item
Given any non-decreasing function $\phi:[0,\ell') \lra [0,\ell)$,
the curve $\xi \circ \phi$ is self-contracted.

\item
If $\xi(t_1)=\xi(t_3)$ for some $t_1<t_3$,
then we have $\xi(t_2)=\xi(t_1)$ for all $t_2 \in [t_1,t_3]$.
\end{enumerate}
\end{lemma}

The function $\phi$ in (i) above is not necessarily continuous nor injective.
One can also consider a self-contracted curve defined on a disconnected set,
for example, $\xi:[0,\ell] \cup [\ell',\ell'') \lra X$.
In this case, however, the extension $\hat{\xi}:[0,\ell'') \lra X$
defined by $\hat{\xi}(t):=\xi(\ell)$ (or $\xi(\ell')$) for $t \in (\ell,\ell')$ is self-contracted.
Therefore considering only intervals does not lose any generality.

\subsection{Quasi-convex functions}\label{ssc:qc}

The remainder of the section is devoted to the Euclidean setting.
We will denote by $\|\cdot\|$ and $\langle \cdot,\cdot \rangle$
the Euclidean norm and inner product, respectively.
A function $f:\R^n \lra \R$ is said to be \emph{quasi-convex} if
\begin{equation}\label{eq:qc}
f\big( (1-s)x+sy \big) \le \max\{f(x),f(y)\} \qquad
 \text{for all}\ x,y \in \R^n,\, s \in (0,1).
\end{equation}
This is equivalent to the property that the sub-level set $\{x \in \R^n \,|\, f(x) \le a\}$
is convex for every $a \in \R$.
We observe from the following examples that
quasi-convexity is a weaker and much more flexible condition than convexity.

\begin{example}[Quasi-convex functions]\label{ex:ESC}
\begin{enumerate}[(a)]
\item
Any convex function $f:\R^n \lra \R$ is clearly quasi-convex.

\item
When $n=1$, then any monotone non-decreasing (or non-increasing) functions
are quasi-convex.

\item
Again in the case of $n=1$,
a function $f:\R \lra \R$ such that $f|_{(-\infty,0]}$ is non-increasing
and that $f|_{[0,\infty)}$ is non-decreasing is quasi-convex.
\end{enumerate}
\end{example}

When $f:\R^n \lra \R$ is $C^1$ and quasi-convex,
then any \emph{gradient curve} $\xi:[0,\ell) \lra \R^n$ of $f$
(that is to say, a solution to $\dot{\xi}(t)=-\nabla f(\xi(t))$) is self-contracted.
Indeed, given $T \in (0,\ell)$ and any $t \in [0,T)$, we have
\begin{align*}
\frac{d}{dt}\Big[ \|\xi(T)-\xi(t)\|^2 \Big]
&= -2\langle \dot{\xi}(t),\xi(T)-\xi(t) \rangle
 = 2\big\langle \nabla f\big( \xi(t) \big),\xi(T)-\xi(t) \big\rangle \\
&= 2\lim_{s \downarrow 0}\frac{f((1-s)\xi(t)+s\xi(T)) -f(\xi(t))}{s}
 \le 0
\end{align*}
since $f(\xi(T)) \le f(\xi(t))$.
We will discuss a more general situation
(lower semi-continuous quasi-convex functions on $\CAT(0)$-spaces)
in Proposition~\ref{pr:CATsc},
see also \S \ref{ssc:gf'} for a relation with the evolution variational inequality.

\begin{remark}\label{rm:QC}
We remark that the self-contractedness of gradient curves does not imply quasi-convexity.
In fact, any $C^1$-function on $\R$ satisfies this property.
In order to characterize convexity, we need a condition,
not only on the behavior of a single curve,
but on a pair of curves (like the contraction property \eqref{eq:cont})
or on a pair of a point and a curve (like the evolution variational inequality \eqref{eq:EVI}).
\end{remark}

\subsection{Angle estimates}\label{ssc:key}

The following simple inequality, going back to \cite[(1) in \S 2]{MP},
plays a fundamental role in the study of self-contracted curves.
See also \cite[Lemma~2.7]{D3L}, \cite[\S 3.3]{D3R} and \cite[\S 3]{Le}.

\begin{lemma}[Angle estimate]\label{lm:key}
Let $\xi:[0,\ell) \lra \R^n$ be a self-contracted curve.
Then, for each $\tau \in [0,\ell)$ and all $t_1,t_2 \in (\tau,\ell)$ with $\xi(t_1),\xi(t_2) \neq \xi(\tau)$,
we have
\begin{equation}\label{eq:pi/2}
\angle \big( \xi(t_1)-\xi(\tau),\xi(t_2)-\xi(\tau) \big) <\frac{\pi}{2},
\end{equation}
where $\angle(v,w)$ for $v,w \in \R^n \setminus \{0\}$ denotes the Euclidean angle.
\end{lemma}

\begin{proof}
Let $t_1<t_2$ without loss of generality.
Then the self-contractedness \eqref{eq:SC} implies $\|\xi(t_2)-\xi(t_1)\| \le \|\xi(t_2)-\xi(\tau)\|$.
This means that $\xi(t_1) \in \bar{B}(\xi(t_2),\|\xi(t_2)-\xi(\tau)\|) \setminus \{\xi(\tau)\}$,
from which we obtain the claim \eqref{eq:pi/2} (see Figure~\ref{fig1}).
$\qedd$
\end{proof}

\begin{figure}
\centering\begin{picture}(400,200)

\put(200,35){\line(0,1){80}}
\put(200,35){\line(1,1){50}}

\put(190,20){$\xi(\tau)$}
\put(190,123){$\xi(t_2)$}
\put(245,93){$\xi(t_1)$}

\put(199,34){\rule{2pt}{2pt}}
\put(199,114){\rule{2pt}{2pt}}
\put(249,84){\rule{2pt}{2pt}}

\thicklines
\put(200,115){\circle{160}}
\qbezier(200,50)(207,50)(210.5,45.5)

\end{picture}
\caption{Proof of Lemma~\ref{lm:key}}\label{fig1}
\end{figure}

The property described in Lemma~\ref{lm:key} characterizes self-contracted curves
among $C^1$-curves.
Indeed, if $\xi$ is $C^1$, then \eqref{eq:pi/2} yields for all $t \in (0,T)$
\[ \frac{d}{dt}\Big[ \|\xi(T)-\xi(t)\|^2 \Big]
 =-2 \|\dot{\xi}(t)\| \|\xi(T)-\xi(t)\| \cos\angle\big( \dot{\xi}(t),\xi(T)-\xi(t) \big) \le 0 \]
(by letting $t=\tau$ and $T=t_2$).
Therefore $\xi$ is self-contracted.

The following fundamental geometric lemma (cf.\ \cite[Lemma~3.2]{D3L})
enables us to control the \emph{radius} of the set of directions
$\xi(t)-\xi(\tau)$ for $t>\tau$.
We stress that this step is not dimension-free.
Let us give a proof along \cite{D3L} for completeness.
Denote by $\Sph^{n-1} \subset \R^n$ the unit sphere equipped with
the angle (intrinsic) distance $\angle$.

\begin{lemma}[Radii of sets of diameter $\le \pi/2$]\label{lm:rad}
There exists a constant $\theta_n \in [0,\pi/2)$ depending only on the dimension $n$ such that,
for any subset $\Delta \subset \Sph^{n-1}$ with $\diam_{\angle}(\Delta) \le \pi/2$,
we can find some $\bar{v} \in \Sph^{n-1}$ for which
$\Delta \subset \bar{B}_{\angle}(\bar{v},\theta_n)$ holds.
\end{lemma}

We will see in the proof that one can take $\theta_n=\arccos[(2 \cdot 3^n)^{-1}]$.
In low dimensions this estimate can be improved by a direct argument,
for instance, clearly $\theta_1=0$ and $\theta_2=\pi/4$.
We put $\angle$ in the subscripts to clarify that the angle distance is employed,
namely $\diam_{\angle}(\Delta)=\sup_{v,w \in \Delta} \angle(v,w)$.

\begin{proof}
The idea of the proof is to construct a ``barycenter'' of $\Delta$.
Since $\Delta$ may not be uniformly distributed,
we first take an arbitrary maximal set $\{v_i\}_{i=1}^m \subset \Delta$ fulfilling the condition:
\[ \angle(v_i,v_j) \ge \frac{\pi}{3} \qquad \text{if}\ i \neq j. \]
Such a set has cardinality at most $3^n$ (namely $m \le 3^n$)
by the standard argument due to the (metric) doubling condition (see \cite[Lemma~3.1]{D3L}).
Given $w \in \Delta$ one can choose $i_0$ such that $\angle(w,v_{i_0})<\pi/3$.
Combining this with the hypothesis $\diam_{\angle}(\Delta) \le \pi/2$, we have
\[ \bigg\langle w,\sum_{i=1}^m v_i \bigg\rangle
 \ge \langle w,v_{i_0} \rangle >\frac{1}{2}. \]
Thus $\sum_{i=1}^m v_i \neq 0$ and we put
\[ \bar{v} := \bigg\| {\sum_{i=1}^m v_i} \bigg\|^{-1} \cdot \sum_{i=1}^m v_i \in \Sph^{n-1}. \]
Together with the trivial bound $\| {\sum}_{i=1}^m v_i \| \le m \le 3^n$,
we conclude that
\[ \langle w,\bar{v} \rangle
 > \bigg( 2 \cdot \bigg\| {\sum_{i=1}^m} v_i \bigg\| \bigg)^{-1}
 \ge \frac{1}{2 \cdot 3^n} \]
as desired.
$\qedd$
\end{proof}

The above lemma explains how the dimension of the space comes into play.
Let us compare Lemma~\ref{lm:rad} with the following example in the same spirit as \cite[Example~2.2]{D3R}
(a counter-example to the rectifiability in an infinite-dimensional space).

\begin{example}[Infinite dimensional case]\label{ex:infd}
Consider the curve $\xi:[0,\infty) \lra L^2(\R)$ defined by $\xi(t):=f_t$, where
\[ f_t(x) := \begin{cases} 1 & \text{for}\ x \in [t,t+1], \\ 0 & \text{for}\ x \not\in [t,t+1]. \end{cases} \]
This is continuous, self-contracted and bounded ($\|f_t\|_{L^2}=1$ for all $t \ge 1$),
whereas
\[ \sL(\xi) \ge \sum_{i=1}^{\infty} \|f_i-f_{i-1}\|_{L^2}
 =\sum_{i=1}^{\infty} \sqrt{2} =\infty. \]
Observe also that $\langle f_i,f_j \rangle_{L^2}=0$ for any distinct $i,j \in \N$,
while the analogue to Lemma~\ref{lm:rad} does not hold since
$\lim_{i \to \infty} \langle f,f_i \rangle_{L^2} =0$ for all $f \in L^2(\R)$.
\end{example}

\subsection{Rectifiability}\label{ssc:Erect}

Lemmas~\ref{lm:key} and \ref{lm:rad} show that
the assumption in \cite[Claim~2]{D3L} is fulfilled.
Then we can follow the lines of \cite[\S 3]{D3L} (Claims~1, 2) to prove the rectifiability.
To this end, we shall derive from the lemmas that the size of the trajectory
\[ \Xi(t) := \{ \xi(s) \,|\, s \in [t,\ell) \}, \]
measured in the direction given by Lemma~\ref{lm:rad}, is (uniformly) decreasing.
Given $v \in \Sph^{n-1}$, we define $P_v:\R^n \lra \R$ as
the orthogonal projection to the line in the direction $v$ (identified with $\R$),
namely $P_v(x):=\langle x,v \rangle$.
Then, for $\Omega \subset \R^n$, we define $\Pi_v(\Omega) \subset \R$
as the closed convex hull of $P_v(\Omega)$
(in other words, the smallest closed interval including $P_v(\Omega)$).

\begin{lemma}[Directional decrease]\label{lm:claim2}
Let $\xi:[0,\ell) \lra \R^n$ be a self-contracted curve, $\tau \in [0,\ell)$ and
$\bar{v}_{\tau} \in \Sph^{n-1}$ be given by Lemma~$\ref{lm:rad}$ for
\[ \Delta =\bigg\{ \frac{\xi(t)-\xi(\tau)}{\|\xi(t)-\xi(\tau)\|}
 \,\bigg|\, t \in (\tau,\ell),\, \xi(t) \neq \xi(\tau) \bigg\} \subset \Sph^{n-1}. \]
Then, for any $T \in (\tau,\ell)$ and
$v \in \Sph^{n-1}$ with $\|v-\bar{v}_{\tau}\| \le \ve_n$, we have
\begin{equation}\label{eq:WW}
\big| \Pi_v \big( \Xi(T) \big) \big|
 \le \big| \Pi_v \big( \Xi(\tau) \big) \big|
 -\ve_n \|\xi(T)-\xi(\tau)\|,
\end{equation}
where $|\cdot|$ denotes the $1$-dimensional Lebesgue measure
and $\ve_n=(\cos \theta_n)/3=(2 \cdot 3^{n+1})^{-1}$.
\end{lemma}

Recall from Lemma~\ref{lm:key} that $\diam_{\angle}(\Delta) \le \pi/2$,
thereby Lemma~\ref{lm:rad} applies.
We took the convex hull of trajectories since $\xi$ may not be continuous.

\begin{proof}
Note that $|\Pi_v(\cdot)|$ is invariant under parallel translations,
thus we consider $\Xi(T)-\xi(\tau)$ instead of $\Xi(T)$.
Since there is nothing to prove if $\xi(T)=\xi(\tau)$, we assume $\xi(T) \neq \xi(\tau)$.
Then it follows from Lemma~\ref{lm:dcon}(ii) that $\xi(t) \neq \xi(\tau)$ for all $t \in [T,\ell)$.
By (the proof of) Lemma~\ref{lm:rad}, we find
\[ \langle \xi(t)-\xi(\tau),\bar{v}_{\tau} \rangle
 >3\ve_n \|\xi(t)-\xi(\tau)\| \]
for all $t \in [T,\ell)$.
Moreover, since $\|\xi(t)-\xi(T)\| \le \|\xi(t)-\xi(\tau)\|$ by the self-contractedness, we have
\begin{equation}\label{eq:tT}
 \|\xi(t)-\xi(\tau)\| \ge \frac{\|\xi(t)-\xi(\tau)\| +\|\xi(t)-\xi(T)\|}{2}
 \ge \frac{\|\xi(T)-\xi(\tau)\|}{2}
\end{equation}
(see Figure~\ref{fig2}).
Hence we have,
for $v \in \Sph^{n-1}$ with $\|v-\bar{v}_{\tau}\| \le \ve_n$,
\begin{align*}
\langle \xi(t)-\xi(\tau),v \rangle
&\ge \langle \xi(t)-\xi(\tau),\bar{v}_{\tau} \rangle - \|\xi(t)-\xi(\tau)\| \cdot \|v-\bar{v}_{\tau}\| \\
&> (3\ve_n -\ve_n) \|\xi(t)-\xi(\tau)\|
 \ge \ve_n \|\xi(T)-\xi(\tau)\|.
\end{align*}
This means that
$\Pi_v(\Xi(T)-\xi(\tau)) \subset (\ve_n \|\xi(T)-\xi(\tau)\|,\infty)$,
while $0 \in \Pi_v(\Xi(\tau)-\xi(\tau))$ and clearly $\Xi(T) \subset \Xi(\tau)$.
Therefore we obtain the claim \eqref{eq:WW}.
$\qedd$
\end{proof}

\begin{figure}
\centering\begin{picture}(400,150)

\put(200,35){\line(4,1){160}}
\put(200,35){\line(-4,1){160}}

\put(200,35){\vector(0,1){95}}

\put(190,20){$\xi(\tau)$}
\put(190,135){$\bar{v}_{\tau}$}
\put(291,93){$\xi(T)$}
\put(252,105){$\xi(t)$}
\put(180,53){$\theta_n$}

\put(199,34){\rule{2pt}{2pt}}
\put(299,84){\rule{2pt}{2pt}}
\put(259,95){\rule{2pt}{2pt}}

\put(254,72){\line(1,-2){4}}
\put(254,72){\line(-2,-1){8}}

\qbezier(200,50)(189,48)(185,39)

\thicklines
\put(230,100){\line(1,-2){40}}
\qbezier[40](200,35)(250,60)(300,85)

\end{picture}
\caption{Proof of Lemma~\ref{lm:claim2}}\label{fig2}
\end{figure}

The estimate \eqref{eq:WW} tells that $\Xi(T)$ is ``smaller'' than $\Xi(\tau)$
in the directions close to $\bar{v}_{\tau}$.
Since $\bar{v}_{\tau}$ depends on $\tau$, we take the average in directions,
called the \emph{mean width}, as follows:
\begin{equation}\label{eq:W}
\sW(\Omega) := \frac{1}{\bA(\Sph^{n-1})} \int_{\Sph^{n-1}} |\Pi_v(\Omega)| \,\bA(dv)
\end{equation}
for $\Omega \subset \R^n$, where $\bA$ denotes the standard measure on $\Sph^{n-1}$.
Clearly $\sW(\Omega) \le \diam(\Omega)$ holds.
Now we are ready to complete the proof of the rectifiability of $\xi$.

\begin{theorem}[Rectifiability in $\R^n$, \cite{D3L}]\label{th:Erect}
Let $\xi:[0,\ell) \lra \R^n$ be a self-contracted curve.
Then we have
\[ \sL(\xi) \le C_n \cdot \sW\big( \Xi(0) \big), \]
where $\Xi(0)=\xi([0,\ell))$ is the image of $\xi$ and $C_n \ge 1$ depends only on $n$.
In particular, we have $\sL(\xi) \le C_n \cdot \diam(\Xi(0))$.
\end{theorem}

\begin{proof}
The assertion is void when $\Xi(0)$ is unbounded,
thus we assume $\diam(\Xi(0))<\infty$.
We first take the average of \eqref{eq:WW} to obtain an estimate of the mean width.
For $\tau,T$ and $\bar{v}_{\tau}$ as in Lemma~\ref{lm:claim2}, put
$\Sigma_{\tau}:=B_{\|\cdot\|}(\bar{v}_{\tau},\ve_n) \cap \Sph^{n-1}$.
Then it follows from \eqref{eq:WW} and $\Xi(T) \subset \Xi(\tau)$ that
\[ \sW\big( \Xi(T) \big) \le \sW\big( \Xi(\tau) \big)
 - \frac{\bA(\Sigma_{\tau})}{\bA(\Sph^{n-1})} \cdot \ve_n \|\xi(T)-\xi(\tau)\|. \]
Since $\bA(\Sigma_{\tau})$ is independent of $\tau$,
we will denote it by $a_n$.
Given an arbitrary partition $0=t_0<t_1<t_2< \cdots <t_k<\ell$,
we apply the above estimate to $(\tau,T)=(t_{i-1},t_i)$ and find
\[ \frac{a_n}{\bA(\Sph^{n-1})} \cdot \ve_n \sum_{i=1}^k \|\xi(t_i)-\xi(t_{i-1})\|
 \le \sW\big( \Xi(0) \big) -\sW\big( \Xi(t_k) \big) \le \sW \big( \Xi(0) \big). \]
Taking the supremum over all partitions, we complete the proof.
$\qedd$
\end{proof}

Notice from the proof that the constant
\[ C_n =\frac{\bA(\Sph^{n-1})}{a_n \ve_n}
 =\frac{(2 \cdot 3^{n+1}) \bA(\Sph^{n-1})}{\bA(B_{\|\cdot\|}(v,(2 \cdot 3^{n+1})^{-1}) \cap \Sph^{n-1})} \]
(with arbitrary $v \in \Sph^{n-1}$) is concretely given and explicitly calculated.

\section{Rectifiability in Hadamard manifolds}\label{sc:Riem}

In the setting of Riemannian manifolds,
we can similarly verify the angle estimate (Lemma~\ref{lm:key}),
whereas the discussion in \S \ref{ssc:Erect} relying on the projection
and the mean width needs to be modified.
The rectifiability of bounded self-contracted curves in Riemannian manifolds
has been shown in \cite[Theorem~2.1]{D3R} by a compactness argument
without any quantitative bound of length.
Here we give a quantitative estimate for the specific class of \emph{Hadamard manifolds}
(complete, simply connected Riemannian manifolds of non-positive sectional curvature).
This provides a perspective toward an extension to $\CAT(0)$-spaces.

Throughout the section except Theorem~\ref{th:Rrect},
let $(M,g)$ be an Hadamard manifold of dimension $n \ge 2$
equipped with the Riemannian distance function $d$.
We will denote by $T_xM$ (resp.\ $U_xM$) the tangent space
(resp.\ the unit tangent sphere) at $x \in M$.
We briefly recall some necessary facts on Hadamard manifolds
(we refer to \cite{Ch} for the basics of comparison Riemannian geometry).
The \emph{Alexandrov--Toponogov triangle comparison theorem} shows that,
for any triplet $x,y,z \in M$ and any minimal geodesic $\gamma:[0,1] \lra M$
from $y$ to $z$, we have
\begin{equation}\label{eq:Hada}
d^2\big( x,\gamma(s) \big) \le (1-s)d^2(x,y) +sd^2(x,z) -(1-s)sd^2(y,z)
\end{equation}
for all $s \in [0,1]$
(this inequality will be employed as the definition of $\CAT(0)$-spaces, see the next section).
In other words, the squared distance function $d^2(x,\cdot)$ is $2$-convex.
We remark that equality holds in \eqref{eq:Hada} in Euclidean spaces.
It follows from \eqref{eq:Hada} that $M$ is contractible
and any two points $x,y \in M$ are connected by a unique minimal geodesic,
that will be denoted by $\gamma_{xy}:[0,1] \lra M$.
Furthermore, the exponential map $\exp_x:T_xM \lra M$ is diffeomorphic
(the \emph{Cartan--Hadamard theorem})
and the inverse map $\exp_x^{-1}$ is well-defined ($\exp_x^{-1}(y)=\dot{\gamma}_{xy}(0)$).

Given $x \in M$ and $y,z \in M \setminus \{x\}$,
we denote by $\angle[yxz]$ the angle between the initial velocities
of the minimal geodesics $\gamma_{xy}$ and $\gamma_{xz}$, namely
$\angle[yxz]:=\angle_x(\dot{\gamma}_{xy}(0),\dot{\gamma}_{xz}(0))$.
Let us also introduce the \emph{Euclidean comparison angle}
$\widetilde{\angle}[yxz] \in [0,\pi]$ for later use, defined by
\begin{equation}\label{eq:Eangle}
\cos \widetilde{\angle}[yxz]
 =\frac{d^2(x,y) +d^2(x,z) -d^2(y,z)}{2d(x,y) d(x,z)}.
\end{equation}
It follows from \eqref{eq:Hada} that $\angle[yxz] \le \widetilde{\angle}[yxz]$,
and we have $\angle[yxz] =\widetilde{\angle}[yxz]$ in $\R^n$.

The following generalization of Lemma~\ref{lm:key} is straightforward.

\begin{lemma}[Angle estimate]\label{lm:Rkey}
Let $\xi:[0,\ell) \lra M$ be a self-contracted curve.
Then, for each $\tau \in [0,\ell)$ and all $t_1,t_2 \in (\tau,\ell)$ with $\xi(t_1),\xi(t_2) \neq \xi(\tau)$,
we have
\[ \angle \big[ \xi(t_1) \xi(\tau) \xi(t_2) \big] <\frac{\pi}{2}. \]
\end{lemma}

\begin{proof}
Assume $t_1<t_2$ without loss of generality
and put $\gamma:=\gamma_{\xi(\tau)\xi(t_1)}$ and $\eta:=\gamma_{\xi(\tau)\xi(t_2)}$.
It follows from \eqref{eq:Hada} and $d(\xi(t_1),\xi(t_2)) \le d(\xi(\tau),\xi(t_2))$ that,
for any $s \in (0,1)$,
\begin{align*}
d^2\big( \gamma(s),\xi(t_2) \big)
&\le (1-s) d^2\big( \xi(\tau),\xi(t_2) \big) +s d^2\big( \xi(t_1),\xi(t_2) \big)
 -(1-s)s d^2\big( \xi(\tau),\xi(t_1) \big) \\
&\le d^2\big( \xi(\tau),\xi(t_2) \big) -(1-s)s d^2\big( \xi(\tau),\xi(t_1) \big).
\end{align*}
Then the claim follows from the first variation formula for the distance function as
\[ 2\langle \dot{\gamma}(0),\dot{\eta}(0) \rangle
 =-\frac{d}{ds} \Big[ d^2\big( \gamma(s),\xi(t_2) \big) \Big] \Big|_{s=0}
 \ge d^2\big( \xi(\tau),\xi(t_1) \big) >0, \]
where $\langle \cdot,\cdot \rangle$ denotes the inner product of $T_{\xi(\eta)}M$
induced from the Riemannian metric $g$.
$\qedd$
\end{proof}

Applying Lemma~\ref{lm:rad} to the tangent space $T_{\xi(\tau)}M$,
we find some unit tangent vector $\bar{v}_{\tau} \in U_{\xi(\tau)}M$ such that
\begin{equation}\label{eq:bv}
\cos \angle_{\xi(\tau)} \big( \bar{v}_{\tau},\exp^{-1}_{\xi(\tau)} \big( \xi(t) \big) \big)
 \ge \frac{1}{2 \cdot 3^n} =3\ve_n
\end{equation}
for all $t>\tau$ with $\xi(t) \neq \xi(\tau)$,
where $\angle_{\xi(\tau)}$ denotes the angle in $T_{\xi(\eta)}M$ induced from $g$.
Different from the Euclidean situation,
besides the \emph{vertical} perturbation of the vector $\bar{v}_{\tau}$ as in Lemma~\ref{lm:claim2},
we need to consider a \emph{horizontal} perturbation as well to discuss the mean width
in the Riemannian setting.
We introduce for this purpose the set
\[ \Omega_{\tau,\sigma} :=\bigg\{ x \in M \,\bigg|\,
 0<d\big( \xi(\tau),x \big) <\sigma,\,
 \angle_{\xi(\tau)} \big( \!\exp^{-1}_{\xi(\tau)}(x),\bar{v}_{\tau} \big)
 \ge \pi -2\arcsin\bigg( \frac{\ve_n}{2} \bigg) \bigg\} \]
for $\sigma>0$.
For each $x \in \Omega_{\tau,\sigma}$, we define
\[ V_x :=\frac{1}{d(x,\xi(\tau))} \exp_x^{-1} \!\big( \xi(\tau) \big) \,\in U_xM, \qquad
\overline{V}\!_x :=\frac{1}{d(x,\xi(\tau))} \exp^{-1}_{\xi(\tau)}(x)  \,\in U_{\xi(\tau)}M \]
(see Figure~\ref{fig3}).
By definition we have
\begin{equation}\label{eq:Vx}
\angle_{\xi(\tau)}(\overline{V}\!_x,\bar{v}_\tau) \ge \pi -2\arcsin\bigg( \frac{\ve_n}{2} \bigg),
\end{equation}
in other words, $\|\bar{v}_{\tau}+\overline{V}\!_x\| \le \ve_n$.

\begin{figure}
\centering\begin{picture}(400,190)

\put(200,135){\line(4,1){160}}
\put(200,135){\line(-4,1){160}}

\put(200,135){\vector(0,1){30}}
\qbezier(192,70)(200,100)(200,135)

\put(177,146){$\xi(\tau)$}
\put(196,172){$\bar{v}_{\tau}$}
\put(262,180){$\xi(t)$}
\put(196,58){$\Omega_{\tau,\sigma}$}
\put(186,61){$x$}
\put(172,82){$V_x$}
\put(179,110){$\overline{V}\!_x$}

\put(199,134){\rule{2pt}{2pt}}
\put(269,170){\rule{2pt}{2pt}}
\put(191,69){\rule{2pt}{2pt}}

\put(200,135){\line(1,-4){28}}
\put(200,135){\line(-1,-4){28}}

\put(210,129){\line(-1,4){2}}
\put(210,129){\line(-4,-1){8}}
\put(190,129){\line(1,4){2}}
\put(190,129){\line(4,-1){8}}

\thicklines
\put(200,135){\line(1,-4){20}}
\put(200,135){\line(-1,-4){20}}
\qbezier(180,55)(200,45)(220,55)

\put(200,135){\vector(0,-1){30}}
\put(192,70){\vector(1,4){6}}

\end{picture}
\caption{$\Omega_{\tau,\sigma}$, $V_x$, $\overline{V}\!_x$}\label{fig3}
\end{figure}

Given $v \in U_xM$, similarly to \S \ref{ssc:Erect},
we define $P_v:T_xM \lra \R$ as the orthogonal projection to $\R v$ (identified with $\R$),
namely $P_v(w):=\langle w,v \rangle$,
and denote by $\Pi_v(\Omega) \subset \R$ the closed convex hull of $P_v(\exp_x^{-1}(\Omega))$
for $\Omega \subset M$.
Clearly $|\Xi_v(\Omega)|=|\Xi_{-v}(\Omega)|$ holds.
We deduce from \eqref{eq:Vx} and \eqref{eq:bv} that,
similarly to Lemma~\ref{lm:claim2},
\begin{equation}\label{eq:PV}
P_{\overline{V}\!_x} \big( {\exp}_{\xi(\tau)}^{-1} \big( \xi(t) \big) \big)
 \le \big\langle {\exp}_{\xi(\tau)}^{-1} \big( \xi(t) \big),-\bar{v}_{\tau} \big\rangle
 +\ve_n d\big( \xi(\tau),\xi(t) \big)
 \le -2\ve_n d\big( \xi(\tau),\xi(t) \big)
\end{equation}
for $0 \le \tau<t<\ell$.
This implies, together with \eqref{eq:tT},
\[ \big| \Pi_{\overline{V}\!_x} \big( \Xi(T) \big) \big|
 \le \big| \Pi_{\overline{V}\!_x} \big( \Xi(\tau) \big) \big|
 -\ve_n d \big( \xi(\tau),\xi(T) \big) \]
for $0 \le \tau <T<\ell$.
We extend this to tangent vectors at $x$ close to $V_x$ as follows.

\begin{lemma}[Directional decrease]\label{lm:Rcl2}
Let $\xi:[0,\ell) \lra M$ be a self-contracted curve and,
given $\tau \in [0,\ell)$ and $\sigma>0$,
define $\Omega_{\tau,\sigma}$, $V$ and $\overline{V}$ as above.
Then, for any $T \in (\tau,\ell)$, $x \in \Omega_{\tau,\sigma}$ and
$v \in U_xM$ with $\|v-V_x\| \le \ve_n$,
we have
\begin{equation}\label{eq:RWW}
\big| \Pi_v \big( \Xi(T) \big) \big|
 \le \big| \Pi_v \big( \Xi(\tau) \big) \big|
 -\frac{\ve_n}{2} d \big( \xi(\tau),\xi(T) \big).
\end{equation}
\end{lemma}

\begin{proof}
Assume $\xi(T) \neq \xi(\tau)$ without loss of generality and fix $t>T$.
We shall estimate $\cos\angle[\xi(\tau) x \xi(t)]$ from below
by means of the non-positive curvature.
We first observe from $\angle \le \widetilde{\angle}$ and \eqref{eq:Eangle} that
\begin{align}
&d\big( x,\xi(t) \big) \cos\angle[\xi(\tau) x \xi(t)]
 +d\big( \xi(\tau),\xi(t) \big) \cos\angle[x \xi(\tau) \xi(t)] \nonumber\\
&\ge d\big( x,\xi(t) \big) \cos\widetilde{\angle}[\xi(\tau) x \xi(t)]
 +d\big( \xi(\tau),\xi(t) \big) \cos\widetilde{\angle}[x \xi(\tau) \xi(t)] \nonumber\\
&= d\big( x,\xi(\tau) \big). \label{eq:VV}
\end{align}
Together with \eqref{eq:PV}, we obtain for any $t>T$
\[ P_{V_x}\big( {\exp}_x^{-1} \big( \xi(t) \big) \big)
 \ge d\big( x,\xi(\tau) \big) -P_{\overline{V}\!_x}\big( {\exp}_{\xi(\tau)}^{-1} \big( \xi(t) \big) \big)
 \ge d\big( x,\xi(\tau) \big) +2\ve_n d\big( \xi(\tau),\xi(t) \big). \]
This implies, for $v \in U_xM$ with $\|v-V_x\| \le \ve_n$,
\begin{align*}
&P_v \big( {\exp}_x^{-1} \big( \xi(t) \big) \big)
 -P_v \big( {\exp}_x^{-1} \big( \xi(\tau) \big) \big)
 = \big\langle {\exp}_x^{-1} \big( \xi(t) \big) -{\exp}_x^{-1} \big( \xi(\tau) \big),v \big\rangle \\
&\ge P_{V_x} \big( {\exp}_x^{-1} \big( \xi(t) \big) \big)
 -P_{V_x} \big( {\exp}_x^{-1} \big( \xi(\tau) \big) \big)
 -\ve_n \big\| {\exp}_x^{-1} \big( \xi(t) \big) -{\exp}_x^{-1} \big( \xi(\tau) \big) \big\|  \\
&\ge (2\ve_n -\ve_n) d\big( \xi(\tau),\xi(t) \big)
 \ge \frac{\ve_n}{2} d\big( \xi(\tau),\xi(T) \big).
\end{align*}
We used the non-positive curvature to see
\[ \big\| {\exp}_x^{-1} \big( \xi(t) \big) -{\exp}_x^{-1} \big( \xi(\tau) \big) \big\|
 \le d\big( \xi(\tau),\xi(t) \big). \]
This completes the proof.
$\qedd$
\end{proof}

Now we are ready to prove the rectifiability.
We remark that, similarly to the Euclidean case,
the possible discontinuity of the curve causes no difficulty nor difference in our proof.

\begin{theorem}[Rectifiability in Hadamard manifolds]\label{th:Hrect}
Let $\xi:[0,\ell) \lra M$ be a self-contracted curve in a Hadamard manifold.
Then we have
\[ \sL(\xi) \le C\big( \Xi(0) \big) \diam\!\big( \Xi(0) \big) <\infty, \]
where $\Xi(0)=\xi([0,\ell))$ and the constant $C(\Xi(0)) \ge 1$ depends only on $n=\dim M$
and the volume of a neighborhood of $\Xi(0)$
$($see \eqref{eq:Rrect} for the precise estimate$)$.
\end{theorem}

\begin{proof}
Fix $\sigma>0$ and $\tau \in [0,\ell)$,
take $\bar{v}_{\tau} \in U_{\xi(\tau)}M$, $\Omega_{\tau,\sigma}$ and $V$ as above,
and let $\Omega$ be the $\sigma$-neighborhood of $\Xi(0)$
(thereby $\Omega_{\tau,\sigma} \subset \Omega$).
We put $\Sigma_x:=B_{\|\cdot\|}(V_x,\ve_n) \cap U_xM$
for $x \in \Omega_{\tau,\sigma}$.
Then, by integrating \eqref{eq:RWW} and noticing  $\Xi(T) \subset \Xi(\tau)$ for $\tau<T$,
we have
\begin{align*}
&\int_{\Omega} \int_{U_xM} \big| \Pi_v \big( \Xi(T) \big) \big|
 \,\bA_x(dv) \bV\!_g(dx) \\
&\le \int_{\Omega} \int_{U_xM} \big| \Pi_v \big( \Xi(\tau) \big) \big|
 \,\bA_x(dv) \bV\!_g(dx)
 -\int_{\Omega_{\tau,\sigma}} \frac{\ve_n}{2} d\big( \xi(\tau),\xi(T) \big) \bA_x(\Sigma_x) \,\bV\!_g(dx),
\end{align*}
where $\bV\!_g$ is the Riemannian volume measure
and $\bA_x$ is the induced measure on $U_xM$.
We remark that $\bA_x(\Sigma_x)$ is independent of $x$ and denote it by $a_n$
(which indeed coincides with $a_n$ in the proof of Theorem~\ref{th:Erect}).
Moreover, $\bV\!_g(\Omega_{\tau,\sigma})$ is bounded below by using a constant $b_n$
depending on $n$ as $\bV\!_g(\Omega_{\tau,\sigma}) \ge b_n \sigma^n$
(by comparing it with the flat Euclidean case).
Now we define a variant of the mean width \eqref{eq:W} as
(with the same symbol by an abuse of notation)
\[ \sW\big( \Xi(t) \big) :=
 \frac{1}{\bA(\Sph^{n-1}) \bV\!_g(\Omega)}
 \int_{\Omega} \int_{U_xM} \big| \Pi_v \big( \Xi(t) \big) \big|
 \,\bA_x(dv) \bV\!_g(dx). \]
Then we find
\[ \sW\big( \Xi(T) \big) \le \sW\big( \Xi(\tau) \big)
 - \frac{a_n b_n \sigma^n}{\bA(\Sph^{n-1}) \bV\!_g(\Omega)}
 \frac{\ve_n}{2} d\big( \xi(\tau),\xi(T) \big). \]
This yields, by the same argument as the proof of Theorem~\ref{th:Erect},
\begin{equation}\label{eq:Rrect}
\sL(\xi) \le \frac{\bA(\Sph^{n-1}) \bV\!_g(\Omega)}{a_n b_n \sigma^n}
 \frac{2}{\ve_n} \sW\big( \Xi(0) \big)
 \le \frac{\bA(\Sph^{n-1}) \bV\!_g(\Omega)}{a_n b_n \sigma^n}
 \frac{2}{\ve_n} \diam\! \big( \Xi(0) \big).
\end{equation}
This completes the proof.
$\qedd$
\end{proof}

Our careful estimate of the constant $C(\Xi(0))$ reveals on what quantities
the length estimate depends (compare this with the compactness argument in \cite{D3R},
see for example Lemma~2.4 in it).
In \eqref{eq:Rrect}, if the Ricci curvature of $M$ is bounded below by some $K<0$,
then $\bV\!_g(\Omega)$ is bounded above by a constant depending on
$n$, $K$ and $\diam(\Xi(0))+\sigma$ (by the Bishop comparison theorem).

Although we do not pursue such a direction in this article for simplicity,
it seems plausible that one can generalize the argument in this section
to general Riemannian manifolds.
Then there are two issues to be dealt with: Positive curvature and cut points.
In order to handle with the positive curvature,
one employs the spherical comparison theorems.
When cut points exist, the following simple lemma can be used
to decompose the manifold into small pieces without cut points.

\begin{lemma}\label{lm:3B}
Let $\xi:[0,\ell) \lra X$ be a self-contracted curve in a metric space $(X,d)$.
If $\xi(t_1),\xi(t_2) \in B(x,r)$ for some $t_1,t_2 \in [0,\ell)$ with $t_1<t_2$,
then $\xi(t) \in B(x,3r)$ for all $t \in (t_1,t_2)$.
\end{lemma}

\begin{proof}
We deduce from the triangle inequality and self-contractedness that
\[ d\big( \xi(t),x \big) < d\big( \xi(t),\xi(t_2) \big) +r
 \le d\big( \xi(t_1),\xi(t_2) \big) +r <3r. \]
$\qedd$
\end{proof}

With this lemma we can extend Theorem~\ref{th:Hrect} to general
(not necessarily simply-connected) Riemannian manifolds of non-positive sectional curvature.

\begin{theorem}[Rectifiability in non-positively curved manifolds]\label{th:Rrect}
Let $(M,g)$ be a complete Riemannian manifold of non-positive sectional curvature,
and $\xi:[0,\ell) \lra M$ be a self-contracted curve whose image is bounded.
Then we have $\sL(\xi)<\infty$.
\end{theorem}

\begin{proof}
Taking Lemma~\ref{lm:3B} into account,
we choose for each $x \in \Xi(0)=\xi([0,\ell))$ an open ball $B(x,r)$
such that there is no pair of cut points in $B(x,3r)$
(hence $B(x,3r)$ is strictly convex and $\CAT(0)$).
Since $\Xi(0)$ is bounded and $M$ is complete, we can extract a finite family of such balls
$B_j=B(x_j,r_j)$, $j=1,2, \ldots, m$, covering $\Xi(0)$.
We set $\widehat{B}_j:=B(x_j,3r_j)$.

By renumbering we can assume $\xi(0) \in B_1$.
Putting $t_1:=\sup\{ t \in [0,\ell) \,|\, \xi(t) \in B_1 \}$,
we deduce from Lemma~\ref{lm:3B} that $\xi(t) \in \widehat{B}_1$ for all $t \in [0,t_1)$.
We set $I_1:=[0,t_1]$ if $\xi(t_1) \in \widehat{B}_1$ (possibly $I_1=\{0\}$),
and $I_1:=[0,t_1)$ otherwise.
Next, if $I_1=[0,t_1)$, then we choose $j$ ($\neq 1$) with $\xi(t_1) \in B_j$.
Otherwise, we take a sequence $\{s_i\}_{i \in \N} \subset (t_1,\ell)$ converging to $t_1$
such that $\xi(s_i) \in B_j$ for some $j$ and all $i \in \N$.
Again by renumbering we can assume $j=2$ in either case.
Then the same argument as the previous step yields the interval $I_2 \subset [0,\ell) \setminus I_1$
(one of the forms $(t_1,t_2]$, $(t_1,t_2)$, $[t_1,t_2]$ and $[t_1,t_2)$)
such that $\xi(t) \in \widehat{B}_2$ for all $t \in I_2$.

Iterating this procedure provides a decomposition
\[ [0,\ell)=I_1 \sqcup I_2 \sqcup \cdots \sqcup I_k \]
for some $k \le m$.
By the construction
we can apply Theorem~\ref{th:Hrect} to each $\xi|_{I_j}$,
thereby $\sL(\xi|_{I_j})<\infty$.
Between $I_j$ and $I_{j+1}$ there may be a jump,
with a length less than or equal to the diameter of $\Xi(0)$.
Since the number of such jumps is at most $k-1$, we conclude that
\[ \sL(\xi) \le \sum_{j=1}^k \sL(\xi|_{I_j}) +(k-1) \diam\!\big( \Xi(0) \big) <\infty. \]
$\qedd$
\end{proof}

Recall that the rectifiability itself is known for general Riemannian manifolds by \cite{D3R}.
Our argument is somewhat more quantitative
thanks to the concrete estimate in Theorem~\ref{th:Hrect}.

\section{Self-contracted curves in $\CAT(0)$-spaces}\label{sc:CAT}

From this section, we take one step forward to a non-smooth setting of \emph{$\CAT(0)$-spaces}.
A $\CAT(0)$-space is a metric space of non-positive sectional curvature
in the sense of triangle comparison theorem.
We refer to \cite{BH,BBI,Jo-book} for the fundamentals of $\CAT(0)$-spaces
and various applications.
Gradient flows of (semi-)convex functions on $\CAT(0)$-spaces are well studied,
see \cite{AGSbook,Ba,Jo,Ma,OP1} among others
as well as \cite{OP2} for a generalization to $\CAT(1)$-spaces
(metric spaces of sectional curvature $\le 1$).

\subsection{$\CAT(0)$-spaces}\label{ssc:CAT}

A metric space $(X,d)$ is said to be \emph{geodesic}
if any pair of points $x,y \in X$ is joined by a continuous curve
$\gamma:[0,1] \lra X$ satisfying $\gamma(0)=x$, $\gamma(1)=y$
and $d(\gamma(s),\gamma(t))=|t-s|d(x,y)$ for all $s,t \in [0,1]$.
We will call such a curve $\gamma$ a \emph{minimal geodesic} from $x$ to $y$.

\begin{definition}[$\CAT(0)$-spaces]\label{df:CAT}
A geodesic metric space $(X,d)$ is called a \emph{$\CAT(0)$-space} if,
for any three points $x,y,z \in X$ and any minimal geodesic
$\gamma:[0,1] \lra X$ from $y$ to $z$, we have
\begin{equation}\label{eq:CAT}
d^2\big( x,\gamma(s) \big) \le (1-s)d^2(x,y) +sd^2(x,z) -(1-s)sd^2(y,z)
\end{equation}
for all $s \in [0,1]$.
\end{definition}

Recall that every Hadamard manifold is a $\CAT(0)$-space.
In fact, a complete Riemannian manifold is a $\CAT(0)$-space
if and only if it is a Hadamard manifold.
By the condition \eqref{eq:CAT},
one can readily verify that every pair $x,y \in X$ is joined by a unique minimal geodesic,
that we will denote by $\gamma_{xy}:[0,1] \lra X$ similarly to the previous section.
Moreover, $(X,d)$ is contractible.
Singular examples of $\CAT(0)$-spaces include trees, books and Euclidean buildings
(see Section~\ref{sc:expl}).
We refer to \cite[\S 9.1.2]{BBI} as well as \cite{CCHO}
for further interesting and important examples,
and to \cite{HH} (among others) for an application to a problem in optimization theory.

Given two nonconstant geodesics $\gamma,\eta:[0,1] \lra X$ emanating
from a common point $x:=\gamma(0)=\eta(0)$,
we can define the \emph{angle} between them at $x$ by
\[ \angle_x(\gamma,\eta)
 :=\lim_{s,t \to 0} \widetilde{\angle}[\gamma(s) x \eta(t)], \]
where $\widetilde{\angle}[yxz]$ is the Euclidean comparison angle defined in \eqref{eq:Eangle}.
The comparison angle $\widetilde{\angle}[\gamma(s) x \eta(t)]$
is monotone non-increasing as $s,t \to 0$,
thereby the limit indeed exists and we have
$\angle_x(\gamma_{xy},\gamma_{xz}) \le \widetilde{\angle}[yxz]$
for all $x \in X$ and $y,z \in X \setminus \{x\}$.
We will also use the notation $\angle[yxz]:=\angle_x(\gamma_{xy},\gamma_{xz})$
compatible with the previous section.

The following first variation formula for the distance function
plays a fundamental role in the study of gradient flows of (semi-)convex functions.
See for instance \cite[Theorem~4.5.6]{BBI} for a proof of the formula.

\begin{theorem}[First variation formula]\label{th:1vf}
Let $(X,d)$ be a $\CAT(0)$-space and take $x \in X$ and $y,z \in X \setminus \{x\}$.
Then we have
\[ \lim_{s \to 0} \frac{d^2(\gamma_{xy}(s),z)-d^2(x,z)}{s}
 =-2d(x,y)d(x,z) \cos \angle[yxz]. \]
\end{theorem}

We close the subsection with a characterization of $\CAT(0)$-spaces,
see \cite{Re} and \cite[II.1.11]{BH}.

\begin{theorem}[Sub-embedding property]\label{th:subem}
A geodesic metric space $(X,d)$ is a $\CAT(0)$-space
if and only if the following \emph{sub-embedding property} holds$:$
For any four points $w,x,y,z \in X$,
there is a quadruplet $\{\tilde{w},\tilde{x},\tilde{y},\tilde{z}\} \subset \R^2$ such that
\begin{align*}
&\|\tilde{x}-\tilde{w}\|=d(w,x), \quad \|\tilde{y}-\tilde{x}\|=d(x,y), \quad
 \|\tilde{z}-\tilde{y}\|=d(y,z), \quad \|\tilde{w}-\tilde{z}\|=d(z,w), \\
&\|\tilde{y}-\tilde{w}\| \ge d(w,y), \quad \|\tilde{z}-\tilde{x}\| \ge d(x,z).
\end{align*}
\end{theorem}

In other words, any quadrilateral $wxyz \subset X$ admits an embedding
$\tilde{w}\tilde{x}\tilde{y}\tilde{z} \subset \R^2$ such that all four edges have the same lengths
and the diagonal edges of $\tilde{w}\tilde{x}\tilde{y}\tilde{z}$
are not shorter than the corresponding edges of $wxyz$.

\subsection{Gradient curves of quasi-convex functions}\label{ssc:gf}

We first recall some fundamental facts on the construction of gradient curves in metric spaces,
for those we refer to the book \cite{AGSbook}.
Let $(X,d)$ be a metric space and $f:X \lra \R$ be lower semi-continuous.
Given $\tau>0$, define the \emph{Moreau--Yosida approximation}
$f_{\tau}$ of $f$ by
\[ f_{\tau}(x):=\inf_{z \in X} \bigg\{ f(z) +\frac{d^2(x,z)}{2\tau} \bigg\}. \]
Then we define
\[ \cJ^f_{\tau}(x) :=\bigg\{ z \in X \,\bigg|\,
 f(z)+\frac{d^2(x,z)}{2\tau} =f_{\tau}(x) \bigg\}. \]
A point in $\cJ^f_{\tau}(x)$ is regarded as an approximation of
the point on the gradient curve of $f$ at time $\tau$ from $x$.
In this manner one can construct a \emph{discrete-time gradient curve} as follows:
\begin{equation}\label{eq:disc}
x_{\bm{\tau}}^0:=x_0\ \text{and recursively choose arbitrary}\ 
 x_{\bm{\tau}}^k \in \cJ^f_{\tau_k}(x_{\bm{\tau}}^{k-1})
 \ \text{for}\ k \in \N,
\end{equation}
where $\bm{\tau}:=\{\tau_k\}_{k \in \N}$ is a sequence of positive numbers.

Now we consider a $\CAT(0)$-space $(X,d)$
and shall see that discrete-time gradient curves of quasi-convex functions are self-contracted
(Proposition~\ref{pr:CATsc}).
This extends the Euclidean result in \cite[Proposition~4.16]{D3L}.
The quasi-convexity \eqref{eq:qc} is naturally generalized to this setting:
A function $f:X \lra \R$ is said to be \emph{quasi-convex} if we have
\[ f\big( \gamma_{xy}(s) \big) \le \max\{ f(x),f(y) \} \qquad
 \text{for all}\ x,y \in X,\, s \in (0,1). \]
A related notion of \emph{$\lambda$-convexity} for $\lambda \in \R$ is defined by
\[ f\big( \gamma_{xy}(s) \big) \le (1-s)f(x) +sf(y) -\frac{\lambda}{2}(1-s)s d^2(x,y)
 \quad \text{for all}\ x,y \in X,\, s \in (0,1). \]
We say that $f$ is \emph{semi-convex} if it is $\lambda$-convex for some $\lambda<0$.
Recall that the $\CAT(0)$-property \eqref{eq:CAT} is understood as the $2$-convexity of
the squared distance function $d^2(x,\cdot)$.

Let us begin with an auxiliary lemma on the well-posedness
of discrete-time gradient curves.

\begin{lemma}\label{lm:CATqc}
Let $(X,d)$ be a complete $\CAT(0)$-space and $f:X \lra \R$ be
a lower semi-continuous quasi-convex function.
Assume in addition that $f$ satisfies one of the following two conditions$:$
\begin{enumerate}[{\rm (1)}]
\item $\inf_X f >-\infty;$
\item $f$ is $\lambda$-convex for some $\lambda<0$.
\end{enumerate}
Then, for any $x \in X$ and $\tau>0$ $(\tau< (-\lambda)^{-1}$ in the case of $(2))$,
$\cJ^f_{\tau}(x)$ is nonempty.
Moreover, $\cJ^f_{\tau}(x)$ consists of a single point if $(2)$ holds
and $\tau< (-\lambda)^{-1}$.
\end{lemma}

\begin{proof}
(1)
Notice first that $f_{\tau}(x) \ge \inf_X f >-\infty$.
Take a sequence $\{z_i\}_{i \in \N} \subset X$ such that
\[ \lim_{i \to \infty} \bigg\{ f(z_i) +\frac{d^2(x,z_i)}{2\tau} \bigg\} =f_{\tau}(x). \]
Given any $\ve>0$, choose $N \in \N$ so that
\[ f(z_i) +\frac{d^2(x,z_i)}{2\tau} \le f_{\tau}(x) +\ve
 \qquad \text{for all}\ i \ge N. \]
Since $\inf_X f \le f(z_i) \le f_{\tau}(x) +\ve$,
by extracting a subsequence and letting $N$ larger if necessary,
we can assume that $f(z_i)$ is convergent and
$|f(z_i)-f(z_j)| \le \ve$ holds for all $i,j \ge N$.
Then, for any $i,j \ge N$, we deduce from the quasi-convexity of $f$
and \eqref{eq:CAT} that
\begin{align*}
f_{\tau}(x) &\le f\bigg( \gamma_{z_i z_j} \bigg( \frac{1}{2} \bigg) \bigg)
 +\frac{d^2(x,\gamma_{z_i z_j}(1/2))}{2\tau} \\
&\le \max\{ f(z_i),f(z_j) \}
 +\frac{1}{4\tau} \bigg\{ d^2(x,z_i)+d^2(x,z_j) -\frac{1}{2}d^2(z_i,z_j) \bigg\} \\
&\le \frac{f(z_i)+f(z_j)+\ve}{2}
 +\frac{1}{4\tau} \bigg\{ d^2(x,z_i)+d^2(x,z_j) -\frac{1}{2}d^2(z_i,z_j) \bigg\} \\
&\le f_{\tau}(x) +\ve +\frac{\ve}{2} -\frac{1}{8\tau} d^2(z_i,z_j).
\end{align*}
Hence $d^2(z_i,z_j) \le (12\tau) \cdot \ve$,
thereby $\{z_i\}_{i \in \N}$ is a Cauchy sequence and convergent.
The limit point belongs to $\cJ^f_{\tau}(x)$ thanks to the lower semi-continuity of $f$.

(2)
This is a standard fact, we give an outline of the proof for completeness.
The point is that $\varphi :=f+d^2(x,\cdot)/(2\tau)$
is $(\lambda+\tau^{-1})$-convex,
and $\lambda+\tau^{-1}>0$ by the hypothesis.
The lower semi-continuity of $f$ implies that
$\inf_{B(x,r)} \varphi >\varphi(x)-1 >-\infty$ for sufficiently small $r>0$.
Combining this with the $(\lambda+\tau^{-1})$-convexity,
we find that $\inf_X \varphi >-\infty$.
Then any sequence $\{z_i\}_{i \in \N} \subset X$
such that $\lim_{i \to \infty} \varphi(z_i) =\inf_X \varphi$
is a Cauchy sequence by a similar argument to (1).
Therefore $\cJ^f_{\tau}(x) \neq \emptyset$,
and the uniqueness of a minimizer of $\varphi$
also follows from the $(\lambda+\tau^{-1})$-convexity.
$\qedd$
\end{proof}

We remark that $\cJ^f_{\tau}(x)$ can be empty for a general quasi-convex function $f$.
Moreover, even when $\inf_X f>-\infty$, the set $\cJ^f_{\tau}(x)$
may not be convex.

\begin{example}\label{ex:CATqc}
We consider two simple $1$-dimensional examples.
\begin{enumerate}[(a)]
\item
First, let $f:\R \lra \R$ be $f(x)=-x^3$.
This is quasi-convex (recall Example~\ref{ex:ESC}(b)), however,
$\cJ^f_{\tau}(0)=\emptyset$ since
\[ f_{\tau}(0) =\lim_{z \to \infty} \bigg\{ {-}z^3 +\frac{z^2}{2\tau} \bigg\} =-\infty. \]

\item
We next consider the same function $f(x)=-x^3$ but on $[0,1]$.
Then
\[ \cJ^f_{1/2}(0) =\underset{z \in [0,1]}{\mathrm{argmin}}\, \{-z^3+z^2\} =\{0,1\}. \]
\end{enumerate}
\end{example}

\begin{proposition}[Gradient curves are self-contracted]\label{pr:CATsc}
Let $f:X \lra \R$ be a lower semi-continuous quasi-convex function
on a complete $\CAT(0)$-space $(X,d)$,
and suppose that a discrete-time gradient curve $(x_{\bm{\tau}}^k)_{k \ge 0}$
as in \eqref{eq:disc} exists.
\begin{enumerate}[{\rm (i)}]
\item
The discrete curve $(x_{\bm{\tau}}^k)_{k \ge 0}$ is self-contracted in the sense that
$d(x_{\bm{\tau}}^l,x_{\bm{\tau}}^m) \le d(x_{\bm{\tau}}^k,x_{\bm{\tau}}^m)$
for all $0 \le k<l<m$.
\item
Moreover, the continuous extension
\[ \xi(t):=\gamma_k \bigg( \frac{t-\tau_{k-1}}{\tau_k -\tau_{k-1}} \bigg)
 \quad \text{for}\ t \in [\tau_{k-1},\tau_k),\, k \in \N, \]
where $\gamma_k:=\gamma_{x_{\bm{\tau}}^{k-1} x_{\bm{\tau}}^k}$ and $\tau_0:=0$,
is self-contracted.
\end{enumerate}
\end{proposition}

\begin{proof}
(i)
Fix $k<m-1$ and put $\gamma :=\gamma_{x_{\bm{\tau}}^{k+1} x_{\bm{\tau}}^m}$.
Since $x_{\bm{\tau}}^{k+1} \in \cJ^f_{\tau_{k+1}}(x_{\bm{\tau}}^k)$,
$f(x_{\bm{\tau}}^m) \le f(x_{\bm{\tau}}^{k+1})$ and $f$ is quasi-convex,
we find for all $s \in (0,1)$
\[ f(x_{\bm{\tau}}^{k+1}) +\frac{d^2(x_{\bm{\tau}}^k,x_{\bm{\tau}}^{k+1})}{2\tau_{k+1}}
 \le f\big( \gamma(s) \big) +\frac{d^2(x_{\bm{\tau}}^k,\gamma(s))}{2\tau_{k+1}}
 \le f(x_{\bm{\tau}}^{k+1}) +\frac{d^2(x_{\bm{\tau}}^k,\gamma(s))}{2\tau_{k+1}}. \]
Therefore $d(x_{\bm{\tau}}^k,x_{\bm{\tau}}^{k+1}) \le d(x_{\bm{\tau}}^k,\gamma(s))$,
while the $\CAT(0)$-property \eqref{eq:CAT} shows
\[ d^2\big( x_{\bm{\tau}}^k,\gamma(s) \big)
 \le (1-s)d^2(x_{\bm{\tau}}^k,x_{\bm{\tau}}^{k+1}) +sd^2(x_{\bm{\tau}}^k,x_{\bm{\tau}}^m)
 -(1-s)sd^2(x_{\bm{\tau}}^{k+1},x_{\bm{\tau}}^m). \]
Combining these inequalities and dividing by $s$ yields
\[ d^2(x_{\bm{\tau}}^k,x_{\bm{\tau}}^{k+1}) \le d^2(x_{\bm{\tau}}^k,x_{\bm{\tau}}^m)
 -(1-s) d^2(x_{\bm{\tau}}^{k+1},x_{\bm{\tau}}^m). \]
Letting $s \to 0$ implies
$d(x_{\bm{\tau}}^{k+1},x_{\bm{\tau}}^m) \le d(x_{\bm{\tau}}^k,x_{\bm{\tau}}^m)$,
which completes the proof of (i).

(ii)
Given $k \in \N$, notice that $x_{\bm{\tau}}^k$ attains
the minimum of $d(x_{\bm{\tau}}^{k-1},\cdot)$ in the sub-level set
$Z_k:=\{ x \in X \,|\, f(x) \le f(x_{\bm{\tau}}^k) \}$ (cf.\ \cite[Lemma~4.15]{D3L}).
In other words, $x_{\bm{\tau}}^k$ is the \emph{foot-point} (or the \emph{projection})
of $x_{\bm{\tau}}^{k-1}$ to the convex set $Z_k$.
Then, for any $z \in Z_k$, we deduce from
$d(x_{\bm{\tau}}^{k-1},\gamma_{x_{\bm{\tau}}^k z}(s)) \ge d(x_{\bm{\tau}}^{k-1},x_{\bm{\tau}}^k)$
for all $s$ that
\begin{equation}\label{eq:CATsc}
\angle_{x_{\bm{\tau}}^k}(\gamma_{x_{\bm{\tau}}^k x_{\bm{\tau}}^{k-1}},
 \gamma_{x_{\bm{\tau}}^k z}) \ge \frac{\pi}{2}.
\end{equation}
This together with \eqref{eq:CAT}
implies that $d(\gamma_{x_{\bm{\tau}}^k x_{\bm{\tau}}^{k-1}}(\rho),z)$ is non-decreasing in $\rho$.
Since $\xi(t) \in Z_k$ for all $t \ge \sum_{i=1}^k \tau_i$,
this completes the proof.
$\qedd$
\end{proof}

Let us remark that (ii) in the above proposition was not straightforward from (i),
the estimate \eqref{eq:CATsc} is essential.
In fact, a polygonal curve constructed from a given discrete self-contracted curve
may not be self-contracted.

\begin{remark}[$\CAT(1)$-spaces]\label{rm:CAT}
Along the same lines as \cite{OP2},
one can generalize Proposition~\ref{pr:CATsc} to \emph{$\CAT(1)$-spaces}
with diameter less than $\pi/2$.
In this case, instead of \eqref{eq:CAT},
we have a uniform convexity of the squared distance:
\[ d^2\big( x,\gamma(s) \big) \le (1-s)d^2(x,y) +sd^2(x,z) -c(1-s)sd^2(y,z) \]
for some $c \in (0,1)$ depending on the diameter.
Combining $d(x_{\bm{\tau}}^k,x_{\bm{\tau}}^{k+1}) \le d(x_{\bm{\tau}}^k,\gamma(s))$
in the proof of (i) with the first variation formula (Theorem~\ref{th:1vf})
along $\gamma =\gamma_{x_{\bm{\tau}}^{k+1} x_{\bm{\tau}}^m}$
as well as along $\eta:=\gamma_{x_{\bm{\tau}}^{k+1} x_{\bm{\tau}}^k}$,
we arrive at the modified estimate
\begin{align*}
0 &\le \lim_{s \to 0}
 \frac{d^2(x_{\bm{\tau}}^k,\gamma(s)) -d^2(x_{\bm{\tau}}^k,x_{\bm{\tau}}^{k+1})}{2s}
 =-d(x_{\bm{\tau}}^{k+1},x_{\bm{\tau}}^m) d(x_{\bm{\tau}}^{k+1},x_{\bm{\tau}}^k)
 \cos \angle_{x_{\bm{\tau}}^{k+1}}(\gamma,\eta) \\
&= \lim_{\rho \to 0}
 \frac{d^2(\eta(\rho),x_{\bm{\tau}}^m) -d^2(x_{\bm{\tau}}^{k+1},x_{\bm{\tau}}^m)}{2\rho}.
\end{align*}
Finally, substituting into this the inequality
\begin{align*}
d^2 \big( \eta(\rho),x_{\bm{\tau}}^m \big)
&\le (1-\rho) d^2(x_{\bm{\tau}}^{k+1},x_{\bm{\tau}}^m) +\rho d^2(x_{\bm{\tau}}^k,x_{\bm{\tau}}^m)
 -c(1-\rho)\rho d^2(x_{\bm{\tau}}^{k+1},x_{\bm{\tau}}^k) \\
&\le (1-\rho) d^2(x_{\bm{\tau}}^{k+1},x_{\bm{\tau}}^m) +\rho d^2(x_{\bm{\tau}}^k,x_{\bm{\tau}}^m)
\end{align*}
yields $d(x_{\bm{\tau}}^{k+1},x_{\bm{\tau}}^m) \le d(x_{\bm{\tau}}^k,x_{\bm{\tau}}^m)$
as desired.
The proof of (ii) is the same as the $\CAT(0)$-case.
\end{remark}

\subsection{Gradient curves of convex functions}\label{ssc:gf'}

We briefly comment on what can be derived under the stronger condition of convexity.
Let $(X,d)$ be a complete $\CAT(0)$-space
and $f:X \lra \R$ be a lower semi-continuous convex function.
Then, as the limit of discrete-time gradient curves as $\sup_k |\tau_k| \to 0$,
we obtain a continuous \emph{gradient curve} $\xi:[0,\ell) \lra X$.
One of the most important properties of $\xi$ is the 
\emph{evolution variational inequality}:
\begin{equation}\label{eq:EVI}
\limsup_{\ve \downarrow 0} \frac{d^2(\xi(t+\ve),y) -d^2(\xi(t),y)}{2\ve}
 +f\big( \xi(t) \big) \le f(y) \qquad \text{for all}\ y \in X.
\end{equation}
Then the self-contractedness of $\xi$ immediately follows,
by applying \eqref{eq:EVI} to $y=\xi(T)$ and noticing that $f(\xi(T)) \le f(\xi(t))$ for all $t<T$.

The evolution variational inequality \eqref{eq:EVI}
also implies the \emph{contraction property}:
\begin{equation}\label{eq:cont}
d\big( \xi(t'),\zeta(t') \big) \le d\big( \xi(t),\zeta(t) \big)
\qquad \text{for all}\ t'>t
\end{equation}
along any gradient curves $\xi,\zeta$ of $f$.
This is a useful property when one is looking for a minimizer of $f$
by starting from a random point and tracing a gradient curve.
The self-contractedness would be thought of as a counterpart
to the contraction property \eqref{eq:cont} for a single curve
(though the author could not find any direct connection between these contraction properties).
For a $C^1$-function $f$ on a Euclidean space or a Riemannian manifold,
the contraction property \eqref{eq:cont} is enjoyed by all pairs of gradient curves of $f$
if and only if $f$ is convex.

\begin{remark}[Normed spaces]\label{rm:norm}
It is known by \cite{OSnc} that gradient curves of convex functions
on normed spaces (or Finsler manifolds) do not necessarily satisfy the contraction property.
This is because the first variation formula based on the angle is a genuinely Riemannian property.
In fact, a normed space is being a $\CAT(0)$-space if and only if it is an inner product space.
\end{remark}

\section{Rectifiability in $\CAT(0)$-spaces}\label{sc:CAT+}

In this section, we reconsider the discussion in Section~\ref{sc:Riem}
on Hadamard manifolds step by step, and introduce the sufficient conditions
for the rectifiability of self-contracted curves in $\CAT(0)$-spaces.
In the next section we consider some examples to those our argument applies.

\subsection{Spaces of directions and tangent cones}\label{ssc:cone}

In order to state our conditions for the rectifiability,
we recall two basic notions describing the infinitesimal structures of $\CAT(0)$-spaces
(we refer to \cite{BBI} for more details).
Let $(X,d)$ be a $\CAT(0)$-space and fix $x \in X$.
The set of nonconstant geodesics emanating from $x$ is denoted by $\widetilde{\Sigma}_xX$:
\[ \widetilde{\Sigma}_xX :=\{ \gamma_{xy} \,|\, y \in X \setminus \{x\} \}. \]
Then the angle $\angle_x$ provides a pseudo-distance function of $\widetilde{\Sigma}_xX$.
Define the \emph{space of directions} $\Sigma_xX$ at $x$ as
the completion of the quotient $\widetilde{\Sigma}_xX/\{ \angle_x =0 \}$ with respect to $\angle_x$.
Then $(\Sigma_xX,\angle_x)$ becomes a metric space.
We also define the \emph{tangent cone} $(C_xX,d_x)$ at $x$ as the Euclidean cone over $\Sigma_xX$,
that is to say,
\[ C_xX :=\big( \Sigma_xX \times [0,\infty) \big) \big/ \big( \Sigma_xX \times \{0\} \big), \]
and
\begin{equation}\label{eq:d_x}
d_x\big( (\gamma,s),(\eta,t) \big) :=\sqrt{s^2+t^2 -2st\cos\angle_x(\gamma,\eta)}
\end{equation}
for $(\gamma,s),(\eta,t) \in C_xX$ ($d_x((\gamma,0),(\eta,t))=t$).
The origin of $C_xX$ will be denoted by $o_x:=[(\gamma,0)]$.
We will sometimes identify $\gamma \in \Sigma_xX$ with $(\gamma,1) \in C_xX$.

\subsection{Conditions for rectifiability}\label{ssc:Crect}

Let $(X,d)$ be a complete $\CAT(0)$-space.
First of all, the angle estimate (Lemma~\ref{lm:Rkey}) is readily generalized.

\begin{lemma}[Angle estimate]\label{lm:angle}
Let $\xi:[0,\ell) \lra X$ be a self-contracted curve.
Then, for each $\tau \in [0,\ell)$ and all $t_1,t_2 \in (\tau,\ell)$ with $\xi(t_1),\xi(t_2) \neq \xi(\tau)$,
we have
\[ \angle \big[ \xi(t_1) \xi(\tau) \xi(t_2) \big] <\frac{\pi}{2}. \]
\end{lemma}

\begin{proof}
One can show this in the same manner as Lemma~\ref{lm:Rkey},
thanks to the first variation formula (Theorem~\ref{th:1vf}).
$\qedd$
\end{proof}

The next step was based on Lemma~\ref{lm:rad}
built on the doubling condition of the unit sphere.
One can mimic the proof once the following property is assumed.

\begin{condition}[Total boundedness]\label{cd:m}
Given a bounded set $\Omega \subset X$,
there exists a constant $\bm{m} \in \N$ such that, at any $x \in \Omega$,
every subset $\Delta \subset \Sigma_xX$ satisfying
$\angle_x(\gamma,\eta) \ge \pi/3$ for all distinct $\gamma,\eta \in \Delta$
has cardinality at most $\bm{m}$.
\end{condition}

Recall from the proof of Lemma~\ref{lm:rad} that,
when $X=\R^n$ ($\Sigma_xX=\Sph^{n-1}$),
then one can take $\bm{m}=3^n$.

\begin{lemma}[Radii of sets of diameter $\le \pi/2$]\label{lm:mrad}
Suppose that Condition~{\rm \ref{cd:m}} is satisfied.
Then, for any $x \in \Omega$ and subset $\Delta \subset \Sigma_x X$
with $\diam_{\angle_x}(\Delta) \le \pi/2$,
there exists $\bar{\gamma} \in \Sigma_x X$ for which
\[ \angle_x(\bar{\gamma},\eta) \le \arccos [(2\bm{m})^{-1}] \]
holds for all $\eta \in \Delta$.
\end{lemma}

\begin{proof}
Choose a maximal set $\{\gamma_i\}_{i=1}^k \subset \Delta$
such that $\angle_x(\gamma_i,\gamma_j) \ge \pi/3$ for $i \neq j$.
By Condition~\ref{cd:m} we have $k \le \bm{m}<\infty$.
Let us identify $\gamma \in \Sigma_xX$ with $(\gamma,1) \in C_xX$ in the sequel.
We take the \emph{barycenter} $v \in C_xX$ of the uniform distribution on $\{\gamma_i\}_{i=1}^k$
in the sense that $v$ attains the minimum of the function
\[ C_xX \ni w \ \longmapsto\ \sum_{i=1}^k d_x^2(w,\gamma_i). \]
Such a minimizer uniquely exists and enjoys the \emph{variance inequality}
\begin{equation}\label{eq:var}
\frac{1}{k} \sum_{i=1}^k d_x^2(w,\gamma_i)
 \ge d_x^2(w,v) +\frac{1}{k} \sum_{i=1}^k d_x^2(v,\gamma_i)
\end{equation}
for all $w \in C_xX$ (we refer to \cite{St} for details).
By construction, putting $v=(\gamma_v,s_v)$,
\[ \sum_{i=1}^k d_x^2 \big( (\gamma_v,t),\gamma_i \big)
 =\sum_{i=1}^k \{ t^2+1-2t\cos\angle_x(\gamma_v,\gamma_i) \} \]
attains the minimum at $t=s_v$
(we take any $\gamma_v \in \Sigma_xX$ if $s_v=0$,
though we eventually see that $s_v>0$).
Hence
\[ \sum_{i=1}^k \{ 2s_v -2\cos\angle_x(\gamma_v,\gamma_i) \}=0, \]
which implies
\begin{align}
d_x^2(o_x,v) +\frac{1}{k} \sum_{i=1}^k d_x^2(v,\gamma_i)
&= s_v^2 +\frac{1}{k} \sum_{i=1}^k \{ s_v^2+1-2s_v \cos\angle_x(\gamma_v,\gamma_i) \} \nonumber\\
&= s_v^2 +(1-s_v^2) =1. \label{eq:ov}
\end{align}

Given $\eta \in \Delta$, choose $i_0$ such that $\angle_x(\eta,\gamma_{i_0})<\pi/3$.
Then we deduce from \eqref{eq:var} and \eqref{eq:ov} that
\begin{align*}
d_x^2(o_x,\eta) +d_x^2(o_x,v) -d_x^2(\eta,v)
&\ge 1+s_v^2+\frac{1}{k} \sum_{i=1}^k \{ d_x^2(v,\gamma_i) -d_x^2(\eta,\gamma_i) \} \\
&= 2-\frac{1}{k} \sum_{i=1}^k d_x^2(\eta,\gamma_i) \\
&\ge 2-\frac{1}{k} \{ 2(k-1)+1 \} =\frac{1}{k},
\end{align*}
where we used $d_x^2(\eta,\gamma_i) \le 2$ and $d_x^2(\eta,\gamma_{i_0}) \le 1$
in the latter inequality.
This shows that $v \neq o_x$ ($s_v>0$)
and $2s_v \cos\angle_x(\eta,\gamma_v) \ge k^{-1}$.
Putting $\bar{\gamma}:=\gamma_v$, we have
\[ d_x^2(\eta,\bar{\gamma}) =2-2\cos\angle_x(\eta,\bar{\gamma})
 \le 2-\frac{1}{s_v k} \le 2-\frac{1}{s_v \bm{m}}. \]
Together with
\[ s_v^2 =d_x^2(o_x,v) \le \frac{1}{k} \sum_{i=1}^k d_x^2(o_x,\gamma_i) =1 \]
derived from \eqref{eq:var},
we conclude that $d_x^2(\eta,\bar{\gamma}) \le 2-\bm{m}^{-1}$,
which is equivalent to the claim $\angle_x(\eta,\bar{\gamma}) \le \arccos[(2\bm{m})^{-1}]$.
$\qedd$
\end{proof}

In accordance with the previous sections, we will set
$\bm{\ve} :=(6\bm{m})^{-1}$ in the sequel, thereby
\begin{equation}\label{eq:mrad}
\angle_x(\bar{\gamma},\eta) \le \arccos(3\bm{\ve}).
\end{equation}
In order to discuss the following step concerning projections,
we define $\log_x:X \lra C_xX$ by $\log_x(y):=(\gamma_{xy},d(x,y)) \in C_xX$,
and $P_{\gamma}:C_x X \lra \R$ with $\gamma \in \Sigma_xX$ by
\[ P_\gamma(w) :=s \cdot \cos\angle_x (\gamma,\eta)
 \qquad \text{for}\ w=(\eta,s) \neq o_x, \]
and $P_{\gamma}(o_x):=0$.
Finally, let $\Pi_{\gamma}(\Xi) \subset \R$ be
the closed convex hull of $P_{\gamma} \circ \log_x(\Xi)$ for $\Xi \subset X$.
Then the very same argument as Lemma~\ref{lm:Rcl2} applies.
For $\tau \in [0,\ell)$, let $\bar{\gamma}_{\tau} \in \Sigma_{\xi(\tau)}X$
be given by Condition~{\rm \ref{cd:m}} for
\[ \Delta=\big\{ \gamma_{\xi(\tau)\xi(t)}
 \,\big|\, t \in (\tau,\ell),\, \xi(t) \neq \xi(\tau) \big\} \subset \Sigma_{\xi(\tau)}X \]
with the help of Lemma~\ref{lm:angle}.
Then, for $\sigma>0$, we set
\[ \Omega_{\tau,\sigma} :=\bigg\{ x \in X \,\bigg|\,
 0< d\big( \xi(\tau),x \big) <\sigma,\,
 \angle_{\xi(\tau)} (\gamma_{\xi(\tau) x},\bar{\gamma}_{\tau})
 \ge \pi -2\arcsin \bigg( \frac{\bm{\ve}}{2} \bigg) \bigg\}. \]
We also put, for $x \in \Omega_{\tau,\sigma}$,
\[ V_x :=\gamma_{x \xi(\tau)} \in \Sigma_xX, \qquad
 \overline{V}\!_x :=\gamma_{\xi(\tau) x} \in \Sigma_{\xi(\tau)}X. \]
As an analogue to \eqref{eq:PV}, we find for $t>\tau$ with $\xi(t) \neq \xi(\tau)$
\begin{align}
&P_{\overline{V}\!_x}\big( {\log}_{\xi(\tau)} \big( \xi(t) \big) \big) \nonumber\\
&= -P_{\bar{\gamma}_{\tau}} \big( {\log}_{\xi(\tau)}\big( \xi(t) \big) \big)
 +\big\{{\cos} \angle_{\xi(\tau)}(\overline{V}\!_x,\gamma_{\xi(\tau) \xi(t)})
 +\cos \angle_{\xi(\tau)}(\bar{\gamma}_{\tau},\gamma_{\xi(\tau) \xi(t)}) \big\}
 d\big( \xi(\tau),\xi(t) \big) \nonumber\\
&\le -3\bm{\ve} d\big( \xi(\tau),\xi(t) \big)
 +2\cos\bigg( \frac{\angle_{\xi(\tau)}(\overline{V}\!_x,\gamma_{\xi(\tau) \xi(t)})
 +\angle_{\xi(\tau)}(\bar{\gamma}_{\tau},\gamma_{\xi(\tau) \xi(t)})}{2} \bigg) \nonumber\\
&\qquad \times \cos\bigg( \frac{\angle_{\xi(\tau)}(\overline{V}\!_x,\gamma_{\xi(\tau) \xi(t)})
 -\angle_{\xi(\tau)}(\bar{\gamma}_{\tau},\gamma_{\xi(\tau) \xi(t)})}{2} \bigg)
 d\big( \xi(\tau),\xi(t) \big) \nonumber\\
&\le \bigg\{ {-3}\bm{\ve}
 +2\cos \bigg( \frac{\angle_{\xi(\tau)}(\overline{V}\!_x,\bar{\gamma}_{\tau})}{2} \bigg) \bigg\}
 d\big( \xi(\tau),\xi(t) \big) \nonumber\\
&\le \bigg\{ {-3}\bm{\ve}
 +2\cos \bigg( \frac{\pi}{2} -\arcsin\bigg( \frac{\bm{\ve}}{2} \bigg) \bigg) \bigg\}
 d\big( \xi(\tau),\xi(t) \big) \nonumber\\
&= -2\bm{\ve} d\big( \xi(\tau),\xi(t) \big). \label{eq:PV'}
\end{align}
Here the first inequality follows from \eqref{eq:mrad},
the second from the triangle inequality of $\angle_{\xi(\tau)}$,
and the third from the choice of $\Omega_{\tau,\sigma}$.

\begin{lemma}[Directional decrease]\label{lm:CATcl2}
Let $\xi:[0,\ell) \lra X$ be a self-contracted curve whose image $\xi([0,\ell))$ is included in
$\Omega$ fulfilling Condition~{\rm \ref{cd:m}}.
Then we have, for any $0 \le \tau<T<\ell$, $\sigma>0$, $x \in \Omega_{\tau,\sigma}$
and $\gamma \in \Sigma_xX$ with $\angle_x(V_x,\gamma) \le 2\arcsin(\bm{\ve}/2)$,
\begin{equation}\label{eq:CWW}
\big| \Pi_{\gamma} \big( \Xi(T) \big) \big|
 \le \big| \Pi_{\gamma} \big( \Xi(\tau) \big) \big|
 -\frac{\bm{\ve}}{2} d \big( \xi(\tau),\xi(T) \big).
\end{equation}
\end{lemma}

Recall that $\angle_x(V_x,\gamma) \le 2\arcsin(\bm{\ve}/2)$ is equivalent to
$d_x((V_x,1),(\gamma,1)) \le \bm{\ve}$.

\begin{proof}
The proof follows essentially the same lines as Lemma~\ref{lm:Rcl2}.
Assume $\xi(T) \neq \xi(\tau)$ without loss of generality.
For any $t>T$, we have
\[ P_{V_x}\big( {\log}_x \big( \xi(t) \big) \big)
 \ge d\big( x,\xi(\tau) \big) -P_{\overline{V}\!_x}\big( {\log}_{\xi(\tau)} \big( \xi(t) \big) \big)
 \ge d\big( x,\xi(\tau) \big) +2\bm{\ve} d\big( \xi(\tau),\xi(t) \big) \]
by \eqref{eq:VV} and \eqref{eq:PV'}.
For $\gamma \in \Sigma_xX$ with $\angle_x(V_x,\gamma) \le 2\arcsin(\bm{\ve}/2)$,
we observe by recalling \eqref{eq:d_x} that
\begin{align*}
&P_{\gamma}\big( {\log}_x \big( \xi(t) \big) \big)
 -P_{\gamma}\big( {\log}_x \big( \xi(\tau) \big) \big) \\
&= P_{V_x}\big( {\log}_x \big( \xi(t) \big) \big)
 -P_{V_x}\big( {\log}_x \big( \xi(\tau) \big) \big) \\
&\quad +d\big( x,\xi(t) \big)
 \big\{ {\cos}\angle_x(\gamma,\gamma_{x \xi(t)})
 -\cos\angle_x \big( V_x,\gamma_{x \xi(t)} \big) \big\} \\
&\quad -d\big( x,\xi(\tau) \big)
 \big\{ {\cos}\angle_x(\gamma,\gamma_{x \xi(\tau)})
 -\cos\angle_x \big( V_x,\gamma_{x \xi(\tau)} \big) \big\} \\
&= P_{V_x}\big( {\log}_x \big( \xi(t) \big) \big) -d\big( x,\xi(\tau) \big) \\
&\quad -\frac{1}{2}
 \big\{ d_x^2(\gamma,\gamma_{x \xi(t)}) -d_x^2\big( V_x,\gamma_{x \xi(t)} \big) \big\}
 +\frac{1}{2}
 \big\{ d_x^2(\gamma,\gamma_{x \xi(\tau)}) -d_x^2\big( V_x,\gamma_{x \xi(\tau)} \big) \big\}
\end{align*}
(we identified $\gamma$ with $(\gamma,1) \in C_xX$ in the last line for simplicity).
Thanks to the sub-embedding property of $(C_xX,d_x)$
(Theorem~\ref{th:subem}, obtained as the limit of that in $(X,d)$),
one can take $\Gamma,V,\Gamma_t,\Gamma_{\tau} \in \R^2$
(corresponding to $\gamma, V_x, \gamma_{x \xi(t)}, \gamma_{x \xi(\tau)}$, respectively)
such that
\begin{align*}
\|V-\Gamma\| &=d_x(\gamma,V_x), \quad \|\Gamma_t -V\| =d_x(V_x,\gamma_{x \xi(t)}), \quad
 \|\Gamma_{\tau} -\Gamma_t\| =d_x(\gamma_{x \xi(t)},\gamma_{x \xi(\tau)}), \\
\|\Gamma -\Gamma_{\tau}\| &=d_x(\gamma_{x \xi(\tau)},\gamma), \quad
 \|\Gamma_t -\Gamma\| \ge d_x(\gamma,\gamma_{x \xi(t)}) ,\quad
 \|\Gamma_{\tau} -V\| \ge d_x(V_x,\gamma_{x \xi(\tau)}).
\end{align*}
Observe that
\begin{align*}
\|\Gamma_t -\Gamma\|^2 -\|\Gamma_t -V\|^2
 -\|\Gamma_{\tau} -\Gamma\|^2 +\|\Gamma_{\tau} -V\|^2
&= 2\langle V-\Gamma, \Gamma_t -\Gamma_{\tau} \rangle \\
&\le 2 \| V-\Gamma \| \cdot \| \Gamma_t -\Gamma_{\tau} \|.
\end{align*}
Thus we obtain,
together with $d_x(\gamma_{x \xi(\tau)},\gamma_{x \xi(t)}) \le d(\xi(\tau),\xi(t))$,
\begin{align*}
&d_x^2(\gamma,\gamma_{x \xi(t)}) -d_x^2\big( V_x,\gamma_{x \xi(t)} \big)
 -d_x^2(\gamma,\gamma_{x \xi(\tau)}) +d_x^2\big( V_x,\gamma_{x \xi(\tau)} \big) \\
&\le 2d_x (V_x,\gamma) d_x(\gamma_{x \xi(\tau)},\gamma_{x \xi(t)})
 \le 2\bm{\ve} d\big( \xi(\tau),\xi(t) \big).
\end{align*}
Therefore we conclude, by \eqref{eq:tT},
\[ P_{\gamma}\big( {\log}_x \big( \xi(t) \big) \big)
 -P_{\gamma}\big( {\log}_x \big( \xi(\tau) \big) \big)
 \ge (2\bm{\ve}-\bm{\ve}) d\big( \xi(\tau),\xi(t) \big)
 \ge \frac{\bm{\ve}}{2} d\big( \xi(\tau),\xi(T) \big). \]
This completes the proof.
$\qedd$
\end{proof}

In the final integration part, we do not in general have estimates corresponding to
$a_n$ and $b_n$ in the case of Hadamard manifolds.
Therefore we assume the following conditions,
regarded as the finite-dimensionality and the bounded complexity.
They are mild conditions and satisfied by a rich class of spaces, see the next section.

\begin{condition}[Area ratio]\label{cd:A}
There exists a constant $\bm{a}=\bm{a}(X,\Omega,\bm{\ve})>0$ such that,
for all $x \in \Omega$, we have a measure $\bA_x$ on $\Sigma_xX$ satisfying
\[ \frac{\bA_x(B_{\angle_x}(\gamma,2\arcsin(\bm{\ve}/2)))}
 {\bA_x(\Sigma_xX)} \ge \bm{a} \quad \text{for all}\ \gamma \in \Sigma_xX. \]
\end{condition}

\begin{condition}[Volume ratio]\label{cd:V}
There exists a constant $\bm{b}=\bm{b}(X,\Omega,\bm{\ve},\sigma)>0$
and a measure $\bV$ on $X$ such that
\[ \frac{1}{\bV(\Omega)}
 \bV \bigg( \bigg\{ x \in X \,\bigg|\,
 0<d(z,x)<\sigma,\,
 \angle_z(\gamma_{zx},\gamma) \ge \pi -2\arcsin \bigg( \frac{\bm{\ve}}{2} \bigg) \bigg\} \bigg)
 \ge \bm{b} \]
for all $z \in \Omega$ and $\gamma \in \Sigma_zX$.
\end{condition}

Notice that there is no definite bound in general.
For instance, for trees, $\bm{a}$ is given by the reciprocal of the maximum degree
(see \S \ref{ssc:tree} below).
Condition~\ref{cd:V} fails at points in the boundary of $X$ (if it is nonempty).

\begin{theorem}[Rectifiability in $\CAT(0)$-spaces]\label{th:Crect}
Let $(X,d)$ be a complete $\CAT(0)$-space and $\Omega \subset X$ a bounded set.
Suppose that Conditions~{\rm \ref{cd:m}}, {\rm \ref{cd:A}}, {\rm \ref{cd:V}} are satisfied
on $\Omega$ with some $\bm{m}$, $\bm{\ve}=(6\bm{m})^{-1}$, $\bm{a}$, $\sigma$ and $\bm{b}$.
Then any self-contracted curve $\xi:[0,\ell) \lra X$ such that the $\sigma$-neighborhood
of its image is included in $\Omega$ has finite length.
\end{theorem}

\begin{proof}
By \eqref{eq:CWW} together with Conditions~\ref{cd:A} and \ref{cd:V},
we can estimate the \emph{mean width}
\[ \sW\big( \Xi(t) \big) :=\frac{1}{\bV(\Omega)}
 \int_{\Omega} \bigg( \frac{1}{\bA_x(\Sigma_xX)} \int_{\Sigma_xX}
 \big| \Pi_{\gamma}\big( \Xi(t) \big) \big| \,\bA_x(d\gamma) \bigg) \bV(dx) \]
as
\[ \sW\big( \Xi(T) \big) \le \sW\big( \Xi(\tau) \big)
 -\bm{ab} \cdot \frac{\bm{\ve}}{2} d\big( \xi(\tau),\xi(T) \big). \]
Therefore we obtain
\begin{equation}\label{eq:Crect}
\sL(\xi) \le \frac{2}{\bm{ab\ve}} \sW\big( \Xi(0) \big)
 \le \frac{2}{\bm{ab\ve}}\diam \!\big( \Xi(0) \big) <\infty.
\end{equation}
$\qedd$
\end{proof}

\begin{remark}\label{rm:cond}
All the conditions are \emph{uniform} bounds on a bounded subset $\Omega$ of $X$.
Although we do not have definite bounds in general,
a compactness argument may verify (some of) these conditions
under appropriate assumptions (such as the \emph{geodesic completeness},
in other words, the infinite extendability of geodesics).
We remark that local structures of $\CAT(0)$-spaces are investigated in \cite{Kl}
as well as in a famous unpublished paper by
Otsu--Tanoue (``The Riemannian structure of Alexandrov spaces with curvature bounded above'')
and Lytchak--Nagano's recent preprints \cite{LN1,LN2}.
\end{remark}

\section{Examples}\label{sc:expl}

This section is devoted to several examples of $\CAT(0)$-spaces
to those we can apply the discussion in the previous section.

\subsection{Trees}\label{ssc:tree}

Let $(X,d)$ be a \emph{tree}, namely a graph without loops.
We assume that the maximum of the \emph{degree} (the number of edges emanating from a vertex)
is finite, denoted by $\Lambda_X$.
One can in addition suppose that, without loss of generality,
every geodesic $\gamma:[0,1] \lra X$ can be extended to $\R$ by adding edges to $X$.
The lengths of edges may not be uniform and can be infinite.

Notice that, at each vertex $x$, the angle $\angle_x$ on $\Sigma_xX \times \Sigma_xX$
takes only $0$ and $\pi$.
Hence Condition~\ref{cd:m} is trivially satisfied with $\bm{m}=1$
and \eqref{eq:mrad} holds with $\bm{\ve}=1/3$.
Lemma~\ref{lm:angle} shows that $\Xi(t)=\xi([t,\ell))$ exists only in one side of $\xi(t)$
(the convex hull of $\Xi(t)$ does not include $\xi(t)$ as an interior point).
Thus one can directly obtain
\[ \big| \Pi_{\bar{\gamma}_{\tau}} \big( \Xi(T) \big) \big|
 \le \big| \Pi_{\bar{\gamma}_{\tau}} \big( \Xi(\tau) \big) \big|
 -d\big( \xi(\tau),\xi(T) \big) \]
corresponding to \eqref{eq:CWW},
where $\bar{\gamma}_{\tau} \in \Sigma_{\xi(\tau)}X$ is facing the direction to $\Xi(\tau)$.
Finally, Condition~\ref{cd:A} holds with $\bA_x$ the counting measure
and $\bm{a}=\Lambda_X^{-1}$,
and we have Condition~\ref{cd:V} for the $1$-dimensional Hausdorff measure $\cH^1$
with $\bm{b}=\cH^1(\Omega)^{-1}$ (we put $\sigma=1$).
Therefore we conclude
\[ \sL(\xi) \le 6\Lambda_X \cH^1(\Omega) \diam\!\big( \Xi(0) \big) <\infty \]
by \eqref{eq:Crect}.
We summarize the result of the above discussion as follows.

\begin{proposition}[Rectifiability in trees]\label{pr:tree}
Let $(X,d)$ be a tree with degrees $\le \Lambda_X<\infty$
such that $\cH^1$ is finite on bounded sets.
Then we have, for any bounded self-contracted curve $\xi:[0,\ell) \lra X$,
\[ \sL(\xi) \le 6\Lambda_X \cH^1(\Omega) \diam\!\big( \Xi(0) \big) <\infty,  \]
where $\Omega \subset X$ is the $1$-neighborhood of $\Xi(0)=\xi([0,\ell))$.
\end{proposition}

\begin{remark}[Continuous and discontinuous cases]\label{rm:tree}
We remark that, if $\xi$ is continuous, then Lemma~\ref{lm:dcon}(ii)
shows that the image of $\xi$ is isometric to an interval
and the rectifiability is reduced to the case of the real line $\R$.
For discontinuous self-contracted curves, however,
this is not the case and it is necessary to bound the degree to control the length.
Consider for example the \emph{$k$-spider}
which consists of the $k$-copies of $[0,1]$ bound at $0$,
and let $\xi(t)$ for $t \in [i-1,i)$ be $1$ in the $i$-th leg $[0,1]$.
This discontinuous curve $\xi:[0,k) \lra X$ is self-contracted and $\sL(\xi)=2(k-1)$
(compare this example with Example~\ref{ex:infd}).
\end{remark}

\subsection{Books}\label{ssc:book}

We next consider $2$-dimensional $\CAT(0)$-spaces so-called (open) \emph{books}.
Let
\[ X :=\{ (i,a,b) \,|\, i=1,2,\ldots,k,\, a \in \R,\, b \in [0,\infty) \}/\sim, \]
where $(i,a,b) \sim (j,a',b')$ if $a=a'$ and $b=b'=0$.
Each subset $X_i:=\{i\} \times \R \times [0,\infty)$ is regarded as a sheet,
then $(X,d)$ is a book with $k$ sheets,
bound along the line $L:=[\{i\} \times \R \times \{0\}] \subset X$.
One can alternatively define $X$ as the product
of the $k$-spider (with infinite edge lengths) and $\R$
(see Figure~\ref{fig4}).

\definecolor {gray}{gray}{0.5}
\begin{figure}
\centering\begin{picture}(400,200)

\put(40,80){\line(4,1){300}}
\put(200,120){\line(0,1){80}}
\put(200,120){\line(2,-1){100}}

\qbezier[40](200,120)(210,75)(220,30)
\qbezier[40](200,120)(200,73)(200,26)
\qbezier[40](200,120)(190,75)(180,30)
\qbezier[40](200,120)(175,80)(150,40)

\put(240,165){$X_1$}
\put(310,110){$X_2$}
\put(345,154){$L$}

\put(258,123){$x$}
\put(294,132){$\gamma$}
\put(288,159){$\bar{\gamma}$}
\put(267,182){$\eta$}
\put(259,134){\rule{2pt}{2pt}}

\put(158,165){$(1,0,b)$}
\put(199,167){\rule{2pt}{2pt}}

\put(250,100){$(2,0,b)$}
\put(239,99){\rule{2pt}{2pt}}

\put(140,60){\textcolor{gray}{$X_i$}}

\thicklines
\put(260,135){\vector(4,1){40}}
\put(260,135){\vector(1,4){10}}
\put(260,135){\vector(2,1){35}}

\end{picture}
\caption{Book}\label{fig4}
\end{figure}

In this case, one can again directly verify \eqref{eq:mrad},
with $3\bm{\ve}=\cos(\pi/4)$, by dividing into three cases.
If $x \not\in L$, then $\Sigma_xX$ is isomeric to $\R^2$
and the condition follows from a direct argument.
The same holds if $x \in L$ and $\Delta \subset \Sigma_xX_i$ for some $i$.
Assume finally that $x \in L$ and $\Delta$ is not contained in a single sheet.
Then there is unique $\gamma \in \Sigma_xL$ such that
$\angle_x(\gamma,\eta_1)+\angle_x(\gamma,\eta_2) \le \pi/2$
for all $\eta_1,\eta_2 \in \Delta$.
If $\angle_x(\gamma,\eta) \le \pi/4$ for all $\eta \in \Delta$,
then we can choose $\bar{\gamma}:=\gamma$.
If not, then $\angle_x(\gamma,\eta)>\pi/4$ holds only in a single sheet
$\eta \in \Delta \cap \Sigma_xX_i$.
We put $\theta:=\sup_{\eta \in \Delta \cap \Sigma_xX_i}\angle_x(\gamma,\eta) \in (\pi/4,\pi/2)$
and take $\bar{\gamma} \in \Sigma_xX_i$ with angle $\theta-\pi/4$ from $\gamma$
(see Figure~\ref{fig4} where $i=1$).
Then the claim follows.

Employing the Hausdorff measures $\cH^1$ and $\cH^2$
in Conditions~\ref{cd:A} and \ref{cd:V}, respectively,
we can conclude as follows.

\begin{proposition}[Rectifiability in books]\label{pr:book}
Let $(X,d)$ be a book with $k$-sheets as above.
Then we have, for any bounded self-contracted curve $\xi:[0,\ell) \lra X$,
\[ \sL(\xi) \le Ck \cH^2(\Omega) \diam\!\big( \Xi(0) \big) <\infty,  \]
where $\Omega \subset X$ is the $1$-neighborhood of $\Xi(0)=\xi([0,\ell))$
and $C>1$ is a universal constant.
\end{proposition}

\begin{proof}
By the above discussion and $\bm{\ve}=(3\sqrt{2})^{-1}$, we have
\[ \big| \Pi_{\gamma} \big( \Xi(T) \big) \big|
 \le \big| \Pi_{\gamma} \big( \Xi(\tau) \big) \big|
 -\frac{1}{6\sqrt{2}} d\big( \xi(\tau),\xi(T) \big) \]
in place of \eqref{eq:CWW}.
Conditions~\ref{cd:A} and \ref{cd:V} (with $\sigma=1$) hold with
\[ \bm{a}=\frac{4}{k\pi} \arcsin\bigg( \frac{1}{6\sqrt{2}} \bigg) >\frac{\sqrt{2}}{3k\pi}, \qquad
 \bm{b}=\frac{1}{\cH^2(\Omega)} \cdot 2\arcsin\bigg( \frac{1}{6\sqrt{2}} \bigg)
 >\frac{1}{3\sqrt{2} \cH^2(\Omega)}. \]
This shows the claim with $C=54\sqrt{2}\pi$ by \eqref{eq:Crect}.
$\qedd$
\end{proof}

\subsection{Simplicial complexes}\label{ssc:simcon}

We finally consider a far general class of simplicial complexes.
Let $X$ be a (connected) simplicial complex such that
each $n$-simplex is bi-Lipschitz to the standard $n$-simplex
(in either $\R^n$ or $\R^{n+1}$),
compatible at the intersection of simplexes.
We equip $X$ with the induced length distance $d$.

\begin{theorem}[Rectifiability in simplicial complexes]\label{th:simcom}
Let $(X,d)$ be a simplicial complex as above and suppose the following.
\begin{enumerate}[$(1)$]
\item $(X,d)$ is \emph{locally finite} in the sense that every bounded subset
contains only finitely many vertexes$;$
\item $(X,d)$ is a complete $\CAT(0)$-space$;$
\item $\diam_{\angle_x}(\Sigma_xX)=\pi$ for all $x \in X$.
\end{enumerate}
Then any bounded self-contracted curve $\xi:[0,\ell) \lra X$ has finite length.
\end{theorem}

\begin{proof}
It suffices to verify the three conditions in the previous section.
By virtue of the hypothesis (1), we can assume that the maximal dimension
of simplexes is finite, that will be denoted by $N \in \N$.
Then we introduce the stratification:
\begin{align*}
X_N &:= \bigcup N\text{-simplexes}, \\
X_{N-1} &:= \bigg( \bigcup (N-1)\text{-simplexes} \bigg) \setminus X_N,\ \ldots, \\
X_1 &:= \bigg(  \bigcup 1\text{-simplexes} \bigg) \setminus \bigcup_{n=2}^N X_n.
\end{align*}
Since $(X,d)$ is connected (by the definition of $\CAT(0)$-spaces),
there is no isolated vertex and hence we have
\[ X =X_1 \sqcup X_2 \sqcup \cdots \sqcup X_N. \]
Define
\[ \bV :=\sum_{n=1}^N \cH^n|_{X_n}, \qquad
 \bA_x :=\sum_{n=1}^N \cH^{n-1}|_{\Sigma_x X_n} \quad \text{for}\ x \in X. \]
In addition, let $\delta>0$ be the minimum of the lengths of $1$-simplexes.

Then Conditions~\ref{cd:m} and \ref{cd:A} are verified by dividing
the sets in the claims into strata $\Sigma_xX_n$ and applying the Euclidean estimates,
together with the (local) finiteness of vertexes.
Precisely, given $x \in X_n$,
the diagonal argument on a dense set of $\Sigma_xX_n$
provides a bi-Lipschitz embedding $\Phi:C_xX_n \lra \R^n$
such that $\Phi((\gamma,s))=s \cdot \Phi((\gamma,1))$ for $s \ge 0$.
Let $\rho \ge 1$ be the bi-Lipschitz constant of $\Phi$.
Then, for $\gamma,\eta \in \Sigma_xX_n$ with $\angle_x(\gamma,\eta) \ge \pi/3$,
we have $\|\Phi(\eta)-\Phi(\gamma)\| \ge \rho^{-1}$ and $\|\Phi(\gamma)\| \in [\rho^{-1},\rho]$
(by identifying $\gamma \in \Sigma_xX_n$ with $(\gamma,1) \in C_xX_n$).
Therefore the cardinality as in Condition~\ref{cd:m} is bounded
by a constant depending only on $N$ and $\rho$.
For Condition~\ref{cd:A}, let us observe that,
given $\gamma,\eta \in \Sigma_xX_n$ with $\angle_x(\gamma,\eta) \ge \theta>0$,
we have
\[ \inf_{s>0} \big\| \Phi \big( (\eta,1) \big) -\Phi\big( (\gamma,s) \big) \big\|
 \ge \rho^{-1} \inf_{s>0} d_x\big( (\gamma,s),(\eta,1) \big)
 \ge \rho^{-1} \sin\theta. \]
Combining this with $\|\Phi(\eta)\| \in [\rho^{-1},\rho]$,
we find that
\begin{align*}
\cH^{n-1} \big( B_{\angle_x}(\gamma,\theta) \cap \Sigma_xX_n \big)
&\ge \rho^{-(n-1)} \cdot \cH^{n-1}
 \Big( \Phi \big( B_{\angle_x}(\gamma,\theta) \cap \Sigma_xX_n \big) \Big) \\
&\ge c_n \rho^{-2(n-1)} \cdot
 \cH^{n-1} \Big( B_{\angle}\big( \Phi(\gamma)/\|\Phi(\gamma)\|,\arcsin(\rho^{-2}\sin\theta) \big) \Big),
\end{align*}
where we used
$\Phi(B_{\angle_x}(\gamma,\theta) \cap \Sigma_xX_n) \subset B_{\|\cdot\|}(0,\rho)$
and the projection of $\Phi(B_{\angle_x}(\gamma,\theta) \cap \Sigma_xX_n)$
to $\rho^{-1} \cdot \Sph^{n-1} \subset \R^n$ (which is $1$-Lipschitz in $\|\cdot\|$)
in the second inequality,
and $c_n$ is needed when $x$ is at the boundary of the $n$-simplex.
This estimate gives Condition~\ref{cd:A}.

In order to show Condition~\ref{cd:V},
we take small $\sigma=\sigma(\delta,N)>0$ and observe that the set
\[ \bigg\{ x \in X \,\bigg|\, 0<d(z,x)<\sigma,\,
 \angle_z(\gamma_{zx},\gamma) \ge \pi -2\arcsin \bigg( \frac{\bm{\ve}}{2} \bigg) \bigg\} \]
is nonempty by the hypothesis (3), and its volume is bounded below
by a constant depending on $\sigma,\bm{\ve},N$ and the bi-Lipschitz constants of the simplexes.
This completes the proof.
$\qedd$
\end{proof}

Let us stress that the structural conditions (1), (3) in Theorem~\ref{th:simcom}
are reasonably weak and easily verified in concrete examples.
For instance, \emph{Euclidean} and \emph{hyperbolic buildings} fit this framework
(see, for instance, \cite[\S 12.2]{AB}).

\begin{remark}\label{rm:simcom}
\begin{enumerate}[(a)]
\item
It is obvious that spaces with non-uniform dimensions can satisfy
the hypotheses in Theorem~\ref{th:simcom}.
A typical example is $X:=([0,\infty) \cup \R^2)/\sim$,
where $0 \in [0,\infty)$ and $(0,0) \in \R^2$ are identified.
This is a $\CAT(0)$-space with the length distance,
and verifies the conditions in Theorem~\ref{th:simcom}
(by decomposing $[0,\infty)$ and $\R^2$ into bounded simplexes).

\item
The last condition (3) holds true for geodesically complete spaces (recall Remark~\ref{rm:cond}).
Furthermore, one can replace (3) with the isometric embeddability of $X$
into a geodesically complete $\CAT(0)$-space.
The existence of such an embedding seems an intriguing open problem.
\end{enumerate}
\end{remark}

\section{Further problems}\label{sc:prob}

We discuss three further problems.

\begin{enumerate}[(A)]
\item
The argument in \S \ref{sc:CAT+} would apply to $\CAT(1)$-spaces as well,
due to the argument at the end of Section~\ref{sc:Riem}.
This class covers, for instance, graphs and spherical buildings.

\item
Besides the quasi-convexity,
the theory of convex functions can be generalized in various ways.
If a function $f:\R^n \lra \R$ is $\lambda$-convex for some $\lambda<0$ (recall \S \ref{ssc:gf}),
then we have the corresponding \emph{$\lambda$-evolution variational inequality}:
\[ \limsup_{\ve \downarrow 0} \frac{d^2(\xi(t+\ve),y) -d^2(\xi(t),y)}{2\ve}
 +\frac{\lambda}{2} d^2 \big( \xi(t),y \big) +f\big( \xi(t) \big)
 \le f(y) \]
as well as the \emph{$\lambda$-contraction} of gradient curves:
\[ d\big( \xi(t),\zeta(t) \big) \le \e^{-\lambda t} d\big( \xi(0),\zeta(0) \big) \]
(compare these with \eqref{eq:EVI} and \eqref{eq:cont}).
As a counterpart to the self-contractedness,
the $\lambda$-convexity implies that $\e^{\lambda t/2}d(\xi(t),\xi(T))$
is non-increasing in $t \in [0,T]$.
It is unclear if this property is useful when $\lambda<0$.

\item
It would be also worthwhile to consider the relations with other kinds of convexities,
such as the \emph{$(K,N)$-convexity} studied in \cite{EKS} for $N \ge 1$
and in \cite{Oneg} for $N<0$.
It is readily seen that $(0,N)$-convex functions with $N<0$ are quasi-convex
(see \cite[Definition~2.5]{Oneg}).

\item
The self-contractedness is a weaker and more flexible property
than the usual contraction property.
Therefore it could be applied to spaces other than those treated in this paper.
Here we list some spaces on those the behavior of gradient curves of
(general) convex functions are less understood than $\CAT(0)$-spaces:
Wasserstein spaces (see \cite{AGSbook,Ogra,Sa}), normed spaces (see \cite{OSnc}),
and discrete spaces.
Another interesting class is \emph{Alexandrov spaces}
of curvature bounded below (see \cite{PP,Ly,OP1}).
\end{enumerate}
\bigskip

{\bf Note.}
After completing this paper, Zolotov~\cite{Zo} made an interesting contribution
on the connection between self-contracted curves and the embedding problem of finite snowflakes.
The main tool in his work is Ramsey theory.
The paper \cite{Zo} also includes an extended list of metric spaces
in which bounded self-contracted curves are rectifiable,
including complete, locally compact, geodesically complete $\CAT(1)$-spaces.
This (partially) answers the problem (A) above,
while a quantitative estimate along the strategy in the present paper will be also meaningful.

\renewcommand{\refname}{{\large References}}
{\small

}

\end{document}